\documentclass[aap,preprint]{imsart}

\RequirePackage{amsthm,amsmath,amsfonts,amssymb}
\RequirePackage[numbers]{natbib}
\RequirePackage[colorlinks,citecolor=blue,urlcolor=blue]{hyperref}
\RequirePackage{graphicx}
\usepackage{hyperref,enumitem,tabularx,floatrow,caption,subcaption,comment}

\startlocaldefs

\numberwithin{equation}{section}
\counterwithin*{equation}{section}
\theoremstyle{plain}
\newtheorem{lemma}{Lemma}[section]
\newtheorem{thm}{Theorem}[section]
\newtheorem{defn}{Definition}[section]
\newtheorem{prop}{Proposition}[section]

\newtheorem{assumption}{Assumption}[section]
\newtheorem{remark}{Remark}[section]

\newcommand{\pa}{{\partial}}

\newcommand{\ignore}[1]{}

\DeclareMathOperator*{\argmin}{arg\,min}

\newcommand{\cA}{\mathcal{A}}

\newcommand{\cF}{\mathcal{F}}

\newcommand{\cH}{\mathcal{H}}

\newcommand{\cP}{\mathcal{P}}

\newcommand{\cX}{\mathcal{X}}

\newcommand{\gam}{{\gamma}}
\newcommand{\Gam}{{\Gamma}}

\newcommand{\eps}{{\epsilon}}

\def \a{\alpha}

\def \e{\varepsilon}

\def \1{\mathbf 1}

\def\E{\mathbb{E}}

\def\P{\mathbb{P}}

\def\R{\mathbb{R}}

\newcommand{\comm}[1]{}

\makeatletter
\def\@setcopyright{}
\def\serieslogo@{}
\makeatother
\endlocaldefs

\begin{document}

\begin{frontmatter}
\title{Kyle-Back models with risk aversion and non-Gaussian beliefs}
\runtitle{Kyle-Back models with risk aversion}

\begin{aug}
\author[A]{\fnms{Shreya} \snm{Bose}\ead[label=e1]{sb18m@my.fsu.edu}}
\and
\author[B]{\fnms{Ibrahim} \snm{Ekren}\ead[label=e2]{iekren@fsu.edu}\thanks{Supported in part by NSF Grant DMS 2007826.}}

\address[A]{Department of Mathematics,
Florida State University,
\printead{e1}}

\address[B]{Department of Mathematics,
Florida State University,
\printead{e2}}
\end{aug}

\begin{abstract}
We show that the problem of existence of equilibrium in Kyle's continuous time insider trading model can be tackled by considering a forward-backward system coupled via an optimal transport type constraint at maturity. The forward component is a stochastic differential equation representing an endogenously determined state variable and the backward component is a quasilinear parabolic equation representing the pricing function. By obtaining a stochastic representation for the solution of such a system,  we show the well-posedness of solutions and study the properties of the equilibrium obtained for small enough risk aversion parameter. In our model, the insider has exponential type utility and the belief of the market maker on the distribution of the price at final time can be non-Gaussian. 

\end{abstract}



\end{frontmatter}

\section{Introduction}

The objective of this paper is to establish the existence of equilibrium in financial markets with long-lived asymmetric information. The model we consider is the case with risk averse informed trader of the model introduced in the seminal work of Kyle \cite{kyle} and extended by Back in \cite{ba}. We call these type of models Kyle-Back models. In these models, the price process is obtained via an equilibrium among three types of market participants. The first type of participant is the so-called informed trader who has an informational advantage. The aim of the informed trader is to make the largest expected profit from her superior information while trading against a market maker. Besides the cumulative orders $X_t$ of the informed trader, the market maker also receives orders $Z_t$ from the so-called noise traders. However, the market maker cannot distinguish the orders of the informed trader from the orders of the noise traders. Therefore, the objective of the market maker is to filter the information of the informed trader by only observing the total cumulative order flow $Y_t=X_t+Z_t$ that he receives and the objective of informed trader is to take advantage of his private information while hiding it from the market maker.

Kyle-Back models are canonical models of market microstructure theory and a vast number of extensions of these models have been considered in the literature, \cite{aa2012,ba,back1993,cel,bp,bal,bia,cdf,cc2007,cor2010,cd1}. In particular, similar to the case studied in this paper, \cite{baruch,c,las2007,hs1994,s} study the problem with risk averse informed traders. With the exception of \cite{s}, with risk averse informed traders, these papers either assume that the belief of the market maker on the final value of the price of the asset is Gaussian or only establishes necessary or sufficient conditions on the existence of the equilibrium without being able to check these conditions. The Gaussianity assumption in particular implies that the prices can become negative with positive probability.  Under this assumption, the authors find that the price impact (i.e. sensitivity of the price in changes in total order flow) has to be deterministic which is unrealistic. In fact, given the restrictive definition of equilibrium taken, in \cite{c}, the author shows that, if the informed agent is risk averse, then an equilibrium can only exist if the belief of the market maker on the final value of the asset is Gaussian. Unlike these papers, we do not require a restrictive framework where explicit computations are possible, but exhibit via a fixed point condition both the underlying state and the pricing rule. Then, using this fixed point, we prove that an equilibrium exists for a fairly large class of belief of the market maker and in equilibrium the strategies of the informed trader and the market maker are path-dependent functionals of the total cumulative orders. We also show that this equilibrium can be characterized by a system of backward quasilinear parabolic partial differential equations\footnote{Somehow similar to the fast diffusion type equations in \cite{bl,lo}.} and a forward Fokker-Planck equation coupled via an optimal transport type constraint. 

{
The most challenging problem while establishing the existence of an equilibrium in Kyle-Back models comes from the fact that the informed trader does not only control his position but also, through the filtering problem that the market maker solves, the belief of the market maker on the final value of the price of the asset. With Gaussian beliefs, this filtering problem can be explicitly solved via the Kalman filter and the evolution of the belief of the market maker can be explicitly written. Thus, if the problem is appropriately stated, the problem of the informed agent becomes a simple stochastic control problem of a finite dimensional process. In various extensions of Kyle-Back models, this reduction to a finite dimensional control problem is achieved via the introduction of an auxiliary state process $\xi$ which solves a forward SDE driven by the total cumulative orders $Y$. {In the literature, the dynamics of $\xi$ are described via the so-called weight function. As a condition for the existence of equilibrium, one can derive a system of partial differential equations for the weight function and the pricing rule similar to the ones derived here (see \cite{ccd,c,cn}). One of our main contributions is to choose the correct weight function (or the solution to the system of partial differential equation) by choosing the correct final condition of the system of equations via a fixed point condition.} 

To construct our equilibrium, similarly to the literature, we conjecture that the market maker's pricing rule only depends on a state process $\xi$ whose dynamics are to be determined and which is also observable by all market participants. Under this condition, we first solve the quasilinear pricing PDE and the Fokker-Planck equation associated to the diffusion of $\xi$. These two equations can be interpreted as necessary condition on the weight function for the existence of equilibrium as stated in \cite{cn}. We also require that $\xi$ generates the same filtration as $Y$ so that the utility maximization problem of the informed trader can be stated as a control problem of $\xi$. The novelty of our work is that we do not guess the weight function. Indeed, in the system we study, the final condition of the quasilinear PDE is yet to be determined. We show that this final condition can be determined by an optimal transport type coupling at maturity between the quasilinear PDE and the Fokker Planck equation. This fixed point condition allows us to use the literature in Markov bridges (for example in \cite{ccd}) and to prove the existence of an equilibrium. Thus, our work shows how to use optimal transport theory to combine the well-known necessary conditions of equilibrium on the pricing rules (see \cite{c,las2007,cn}) and the literature on Markov bridges to find an equilibrium. The connection between optimal transport and Kyle-Back models is also explored in \cite{cel}, with a risk-averse market maker (rather than a risk-averse informed trader) and multiple assets. Due to the lack of risk aversion of the informed trader, the paper obtains an equilibrium not based on a fixed point between a PDE and a Fokker-Planck equation, but rather solving the Monge problem and a quadratic PDE.

Our fixed point condition between the Fokker-Planck equation and the quasilinear PDE is somehow reminiscent of mean-field games. However, unlike mean-field games where the coupling between the forward and backward equations is due to the optimal controls of the agents, our coupling is an optimal transport constraint at maturity that determines the final condition of the quasilinear PDE. A major novelty of our fixed point approach is that it allows the use of numerical methods to compute equilibria which is novel in Kyle-Back models. As an example we use Picard's iteration scheme to numerically compute the fixed point for few cases. We observe the convergence of the algorithm and obtain the pricing rules at maturity. }

In order to obtain the existence of such a fixed point we use different tools in optimal transport theory, stochastic analysis and partial differential equations. First of all, using tools based on Levy's parametrix method (\cite{tt}), we provide a solution to the Fokker-Planck equation via the quasilinear pricing PDE. Then, in order to prove the continuous dependence of the solution of the quasilinear pricing PDE on its final condition (which is needed to have a fixed point), we establish a stochastic representation for the solution of a system of quasilinear parabolic equation as the minimizer of an expectation. This representation is novel and is of independent interest.  Since in one dimension, optimal transport maps can be explicitly written using the cumulative distribution functions, we can in fact write the fixed point condition as a condition on the final value of the pricing rule\footnote{With multiple assets, the transport constraint would have required the study of the continuous dependence problem of a Monge-Ampere equation.}. Then, we show the existence of the fixed point via the Schauder fixed point theorem. The compact set where we apply the Schauder fixed point theorem is determined via an application of Caffarelli's contraction theorem \cite{caf,ko}. 

{A similar existence result for same type of system was also obtained in the unpublished PhD thesis \cite{s} under more restrictive assumptions\footnote{We are thankful to an anonymous referee for providing us the link for this thesis.} of boundedness of price but without any condition on risk aversion. Note that unlike in \cite{s}, we show the optimality of the strategy of the insider among all semimartingale strategies and allow $\tilde v$ to be unbounded. More importantly, the optimal transport constraint completely elucidates the coupling between the forward and the backward equations and it is extendable to a multi-asset framework where one can use numerical methods from (multidimensional) optimal transport literature to compute an equilibrium.  }

In the equilibrium we construct, the strategy of the informed trader is to simply force the final value $\xi_T$ of $\xi$ to reach a level that depends on her private information. This can be achieved by considering the appropriate Markov bridge. In fact, the equilibrium strategy of the informed trader is linear in the current value of $\xi$. The pricing rule of the market maker is to quote a price that only depends on the current value of $\xi$ and the solution of the quasilinear pricing PDE. With Gaussian beliefs, we find the same equilibrium as in \cite{c}, and the price impact\footnote{Also called Kyle's Lambda.}, which is the sensitivity of the price in the total demand is deterministic. However, in general, the price impact is a deterministic function of the stochastic process $\xi$ and therefore is stochastic. 

{The pricing rule is a $C^{1,2}(\Lambda)$-path-dependent functional of the paths of $Y$ (see \cite{cn}). In fact, all the path-dependence in the problem is carried by the process $\xi$ which is a pathwise defined, $C^{1,2}(\Lambda)$-functional of $Y$ and all other components of the equilibrium are functions of the current value of $\xi$ (similarly to \cite{ccd}). One can check the structure of path-dependence of the equilibrium satisfies the necessary condition in \cite{cn}. However, unlike \cite{cn}, our purpose is neither to study the structure of path-dependence nor to state necessary and sufficient conditions for the existence of equilibrium. We aim to provide the existence of the equilibrium by constructing the relevant fixed point. }

In equilibrium, the price impact is a supermartingale and its inverse that we call the market depth is a submartingale. We find that the drift part (which is positive) of the market depth has minimal growth for Gaussian beliefs. This means that in average, the market depth will have smaller growth if the belief of the market maker is Gaussian. We argue that with non-Gaussian beliefs the informed trader has to face an additional risk which is the stochasticity of the price impact thus, she will be more conservative in her trading. This allows in equilibrium to have, in average, a larger growth of market depth. 

The rest of the paper is organized as follows. In Section \ref{s.formulation}, we state the Kyle-Back model we aim to solve. Section \ref{s.fp} states the main mathematical results needed to construct the equilibrium. In particular, in Theorem \ref{thm:fixed}, we construct a fixed point which leads to the equilibrium. Section \ref{s.mr} contains the statements regarding the existence of the equilibrium and its properties. In Sections \ref{s.proof1} and \ref{s.construct} we prove the results of the earlier sections. 

\subsection{Notations}
We fix $\a\in (0,1)$ and for $l \geq 0$ denote
$$C_{l,\a}=\{\phi:\R\mapsto \R: 0\leq \pa_{\xi\xi} \phi\leq l, \, \phi(0)=0\}\cap C^{4+\a}_{loc}.$$
We also fix $T>0$, the maturity of the problem and denote $\Lambda=[0,T]\times C([0,T];\R)$ the set of continuous real valued paths on $[0,T]$. 

Let $(\Omega, \cF,(\cF_t)_{t\in [0,T]},\P)$ be a filtered probability space satisfying the usual conditions of right continuity and completeness. 
We assume that for all $t\in [0,T]$, $(\cF_t)$ is the augmentation of the $\sigma-$ algebra generated by $\{\tilde v\}\cup \{B_s:s\leq t\}$ for some $\R-$valued random variable $\tilde v$ and a one dimensional $\cF$-Brownian motion $B$.
{Since $B$ is a $\cF$-Brownian motion and $\tilde v$ is $\cF_0$ measurable,} $B$ and $\tilde v$ are independent.  We denote by $\nu$ the distribution of $\tilde v$, $F_\nu$ the cumulative distribution function of $\nu$, and $p_\nu$ its density function if it exists.

Based on \cite{cont,dup,etz}, we define a {psuedo}metric on $\Lambda$ by 
$d_\infty((t,y),(s,\tilde y))=\sup_{r\in[0,T]}|y_{r\wedge t}-\tilde y_{r \wedge s}| +|t-s|$. 
{ Note that any functional $u:\Lambda\to \R$ that is continuous with respect to $d_\infty$ is in fact non-anticipative. Indeed, for any $(t,y),(t,\tilde y)\in \Lambda$ that are equal up to time $t$, we have that 
$d_\infty((t,y),(t,\tilde y))=0$. Thus, by continuity $u(t,y)=u(t,\tilde y)$ which is the definition of non-anticipativity.}\footnote{ 
The relation between non-anticipativity and adaptedness comes from the fact that if $u\in C(\Lambda)$ is non-anticipative and $Y$ is a process adapted to a given filtration, then $t\mapsto u(t,Y_\cdot)$ is also adapted to the same filtration by  \cite[Theorem 2.7]{cont}.} We say that $u\in C(\Lambda)$ is an element of $C^{1,2}(\Lambda)$ if there exist three processes
$\pa_t u,\pa_y u,\pa_{yy}u\in C(\Lambda)$ so that for all continuous semimartingale $Y$ with bounded characteristics
$$u(t,Y_\cdot)=u(0,0)+\int_0^t \pa_t u(r,Y_\cdot)+\frac{1}{2} \pa_{yy}u(r,Y_\cdot)\frac{d\langle Y\rangle_r }{dt}dr+\int_0^t {\pa_y}u(r,Y_\cdot)dY_r,\, \P\mbox{-a.s.}$$

\section{Problem Formulation}\label{s.formulation}
As in the classical Kyle-Back models \cite{ba,back1993,baruch,cd1,kyle}, we consider a financial market with interest rate $0$. A single asset is traded in continuous time and at time $t=T$ the price of this asset is announced to be $\tilde v$. We denote by $(P_t)$ the price of the asset that will be determined by an equilibrium condition. 

The trading takes place between three groups of participants. The first group is the so-called noise traders. We assume that these agents trade due to some exogenous needs and their cumulative orders is $Z_t=\sigma B_t$ for all $t\in [0,T]$ and for some fixed $\sigma>0$. The second group is the so called informed trader. At time $t=0$, the informed trader learns the final value $\tilde v$ of the asset. We denote by $X_t$ the cumulative orders of the informed trader up to time $t\in [0,T]$. Besides $\tilde v$, the informed trader observes the process $P$. The third group is risk neutral market makers (henceforth "the market maker") who only observes the total orders $Y=X+Z$ and quotes a price for the risky asset. Since, the market maker cannot distinguish the orders coming from the informed trader and the orders coming from the noise traders, the filtration of the market maker $(\cF^m_t)$ is the $\P$-augmentation of the $\sigma$-algebra generated by $Y$ which will be a strict subset of $\cF_t$ for all $t\in[0,T)$ in equilibrium. We now define the class of pricing rules for the market maker.
 \begin{defn}
A pricing rule is a functional $H\in  C^{1,2}(\Lambda)$ that satisfies
\begin{align}\label{eq:admm}
\E\left[H^2(T,\sigma B_\cdot)+\int_0^TH^2(t,\sigma B_\cdot)dt\right]<\infty
\end{align}
and $\pa_yH(t,y_\cdot)>0$ for $(t,y)\in [0,1)\times C([0,T];\R)$. 
The set of pricing rules is denoted $\cH$. 
\end{defn}
{If the quadratic variation of $Y$ is $\sigma^2 t$, then under regularity conditions on $H$, the condition $\pa_yH(t,y_\cdot)>0$ implies that the filtration generated by $Y_t$ and $H(t,Y_\cdot)$ are the same. Indeed, dividing the differential of $H(t,Y_\cdot)$ by $\pa_yH(t,Y_\cdot)$, $Y$ solves the path-dependent SDE 
$$dY_t=-\frac{\pa_t H(t,Y_\cdot)+\frac{\sigma^2\pa_{yy} H(t,Y_\cdot)}{2}}{\pa_y H(t,Y_\cdot)}dt+\frac{1}{\pa_y H(t,Y_\cdot)}d H(t,Y_\cdot)$$
and is therefore adapted to the filtration generated by $H(t,Y_\cdot)$ under some regularity assumptions on $H$}. Thus, we identify that the filtration of the informed trader as $(\cF_t)$\footnote{This statement is the only reason why we introduce path derivatives of functionals.}.

We assume that the informed trader has an exponential utility 
$$-\gamma \exp(-\gamma W_T)$$
where $W_T$ is her gains from trading at final time and $\gamma>0$ is her risk aversion parameter. 
It is shown in \cite{ba} that 
$$W_T=\int_0^T (\tilde v-P_t)dX_t-\langle X,P\rangle_T.$$
We now define the class of admissible strategy for the informed trader.
\begin{defn}
Given the pricing rule $H\in \cH$, an admissible strategy for the informed trader is an $(\cF_t)$ adapted continuous semimartingale $X$ with $X_0=0$ and satisfying 
\begin{align}\label{eq:adi}
\E\left[e^{\gamma\sigma\int_0^T(\tilde v- H(t,X_\cdot+Z_\cdot))dB_t-\frac{\gamma^2\sigma^2}{2}\int_0^T(\tilde v- H(t,X_\cdot+Z_\cdot))^2dt}\right]=1.
\end{align}
We denote by $\cA(H)$ the admissible strategies given $H$. 
\end{defn}

Note that unlike \cite{s,cd1,cdf}, we allow diffusive strategies for the informed trader. In equilibrium any diffusive part of $X$ will be costly for the informed trader and the equilibrium strategy of the informed trader will be absolutely continuous (see \cite{cd} for related questions). However, unlike \cite{ba}, we choose to assume $X$ to be continuous to avoid technicalities due to the presence of jumps. 
{Since $\tilde v$ is $\cF_0$ measurable,  $$e^{\gamma\sigma\int_0^\cdot(\tilde v- H(t,X_\cdot+Z_\cdot))dB_t-\frac{\gamma^2\sigma^2}{2}\int_0^\cdot(\tilde v- H(t,X_\cdot+Z_\cdot))^2dt}$$ is a $\cF$-Doleans-Dade exponential.
The condition \eqref{eq:adi} is a martingality condition for this Doleans-Dade exponential. 
Our decomposition \eqref{eq:expansionutility} hints that if this martingality condition is not satisfied, the utility of the informed trader might degenerate to $0$. 

Our admissibility condition is also a relaxation of the admissibility in \cite{ba}. Indeed, the admissibility condition in \cite{ba} is $\E[\int_0^T H^2(t,X_\cdot+Z_\cdot)dt]<\infty$. However, in \cite{ba}, this condition is only needed to show that $t\mapsto \int_0^t H(s,X_\cdot+Z_\cdot)dB_s$ is a $\cF$-martingale. In this sense, we directly require that the relevant (exponential) local martingale is a martingale rather than giving an integrability condition leading to this martingality.

We show in the Appendix that in the equilibrium that we construct any strategy of the informed trader satisfying the boundedness from below assumption of the realized gains from trading, 
\begin{align}\label{eq:bb}
\inf_{t\in[0,T]} W_t:=\inf_{t\in[0,T]}\int_0^t (\tilde v-P_s)dX_s-\langle X,P\rangle_t\geq C(\tilde v)
\end{align}
for some finite deterministic function $C$, is admissible. 
}
However, we are unable to check that our candidate equilibrium strategy satisfies \eqref{eq:bb} or any of its simple extensions. Thus, we relax the admissibility condition from \eqref{eq:bb} to \eqref{eq:adi} to be able to claim the admissibility of our candidate equilibrium strategy and carry out the proof of its optimality. 

Additionally, if $\nu$ is compactly supported and the pricing rule $H$ is bounded (which is true in all equilibrium if $\nu$ is compactly supported), then
Novikov's condition implies that any continuous $\cF$-semimartingale $X$ with $X_0=0$ is admissible. Thus, our set of admissible strategy contains the set of admissible strategies in \cite[Definition 3]{s}.

We now state the equilibrium condition for the interaction between the informed trader and the market maker which is also given in \cite{ba,c,cn2,cn}.

\begin{defn}
A pair $(H^*,X^*)$ of pricing rule $H^*\in \cH$ and admissible strategy $X^*\in \cA(H^*)$ is an equilibrium if the following conditions are satisfied.
\begin{itemize}
\item If the informed trader uses the strategy $X^*$, the price process 
$P_t=H^*(t,X^*_\cdot+Z_\cdot)$ is rational in the sense that it is a $(\cF^m_t)$ martingale with $P_T=\tilde v$ a.s.
\item If the market maker quotes the price $P_t=H^*(t,Y_\cdot)$ then $X^*$ is a maximizer for the problem 
$$\sup_{X\in \cA(H^*)}\E[-\gamma\exp(-\gamma W_T)|\cF_0].$$
\end{itemize}
\end{defn}
The condition $P_T=\tilde v$ is more restrictive than usual definitions of equilibria in the literature. {The condition is not an admissibility condition for the strategies of the agents but an additional property of the equilibrium that we construct. Thus, we have chosen to include it in our definition. } 

\section{A fixed point problem}\label{s.fp}
In this section, we state results needed to construct an equilibrium. We make the following assumptions on the data of the problem. 
\begin{assumption}\label{assum:main}
We assume either of the following conditions for $\nu$. 
\begin{longlist}
\item $\nu$ has density $p_\nu(x)=\exp(-V(x))$ for some $\kappa$-strongly convex and three times continuously differentiable function $V$ with $\kappa>0$.
\item $\nu$ is compactly supported with $p_\nu$ three times continuously differentiable and bounded from below on the support of $\nu$.
\end{longlist}
\end{assumption}
\begin{remark}
\begin{longlist}
\item Note that under either of the conditions we have $$\int x^2 \nu(dx)<\infty.$$

\item Unfortunately, the lognormal distribution does not satisfy Assumption \ref{assum:main}. However, if we truncate this distribution (ii) is satisfied for any level of truncation.

\item The assumption (i) is made so that we can use the celebrated Caffarelli's contraction theorem \cite{caf,ko} to apply the Schauder fixed point theorem.
\end{longlist}
\end{remark}

In order to exhibit the equilibrium strategies we start with the following quasilinear parabolic PDE. 
\begin{lemma}\label{lem:pdep}
For all $l \in (0,\frac{1}{T\sigma^2\gamma})$ and $\phi\in C_{l,\a}$, there exists a unique $P=P^\phi:[0,T]\times \R\mapsto \R$ of class $C^{1,2}([0,T]\times \R)$ solving the PDE
\begin{align}\label{eq:pdep}
 \pa_tP(t,\xi)+\frac{\sigma^2\pa_{\xi\xi}P(t,\xi)}{2(1-\gamma\sigma^2 \pa_\xi P(t,\xi)(T-t))^2}&=0\\
P(T,\xi)&= \pa_\xi \phi(\xi).
\end{align}
Additionally, $P$ is uniformly Lipschitz continuous in $\xi$ with Lipschitz constant bounded by l, for all $(t,\xi)\in [0,T]\times \R$ the inequalities
\begin{align}\label{eq:lip}
1>\gamma\sigma^2 Tl\geq \gamma\sigma^2 (T-t)\pa_\xi P(t,\xi) >0
\end{align}
are satisfied, and \eqref{eq:pdep} is uniformly elliptic. 
\end{lemma}
\begin{proof}
The proof of the Lemma is given in Section \ref{s.construct}.
\end{proof}

\begin{remark}
By following the ideas in \cite{lo} one can show that the pricing function of an option in a Bachelier market with linear permanent price impact $\lambda>0$ satisfies
$$\pa_t u^{ba}+\frac{\sigma^2}{2}\frac{ \pa_{\xi\xi}u^{ba}}{1-\lambda  \pa_{\xi\xi}u^{ba}}=0.$$ 
Differentiating this equation in $\xi$, we obtain the delta of an option satisfies
$$ \pa_{t}\Delta^{ba}+\frac{\sigma^2}{2}\frac{ \pa_{\xi\xi}\Delta^{ba}}{(1-\lambda \pa_{\xi}\Delta^{ba})^2}=0$$ 
which is \eqref{eq:pdep} for the choice $\lambda= \gamma \sigma^2 (T-t)$. 

Surprisingly, we find that the pricing rule in our model solves the same equation as the delta of an option in a Bachelier model with linear permanent price impact. \footnote{We are grateful to Gregoire Loeper for pointing this out.}
\end{remark}
Given $P^\phi$, we also define the functions
\begin{align*}
\chi=\chi^\phi:(t,\xi)\in[0,T]\times \R\mapsto \R\mbox{ and }\Gamma= \Gamma^\phi:(t,\xi)\in[0,T]\times \R\mapsto \R
\end{align*} 
by the expressions
\begin{align}\label{eq:defzg}
\chi(t,\xi)=&\xi-\gamma\sigma^2 (T-t)P(t,\xi)\\
\Gamma(t,\xi)
=&\int_t^T\frac{\sigma^2 \pa_\xi P(s,0)}{2(1-\gamma \sigma^2(T-s)\pa_\xi P(s,0))}ds\notag\\
&+\int_0^\xi P(t,r)dr-\frac{\gamma\sigma^2 (T-t)}{2}P^2 (t,\xi).\label{eq:constintg}
\end{align}
Thanks to \eqref{eq:lip}, $P$ and $\chi$ are increasing. This point will be crucial to prove that the candidate equilibrium strategy of the informed trader is optimal among semimartingale strategies as discussed and not only among absolutely continuous strategies (see \cite{cd}).

Given these definitions, by a direct computation, we have that 
\begin{align}\label{eq::pp}
P(t,\xi)=\frac{\pa_\xi\Gamma(t,\xi)}{\pa_\xi\chi(t,\xi)}.
\end{align}
Additionally after lengthy computations that are provided in the Appendix, $\Gamma, \chi$ satisfies
\begin{align}\label{eq:system}
 \pa_t\Gamma(t,\xi)+\frac{\sigma^2}{2(\pa_\xi\chi(t,\xi))^2}\pa_{\xi\xi}\Gamma(t,\xi)-\frac{\gamma\sigma^2P^2(t,\xi)}{2}=0\\
\pa_t\chi(t,\xi)+\frac{\sigma^2}{2(\pa_\xi\chi(t,\xi))^2}\pa_{\xi\xi}\chi(t,\xi)-{\gamma\sigma^2P(t,\xi)}=0\label{eq:system2}
\end{align}
with final condition $$\chi(T,\xi)=\xi\mbox{ and }\Gamma(T,\xi)=\phi(\xi).$$
Since, $\phi\in C_{l,\a}$, we also have $\phi(0)=\Gamma(T,0)=0$.
As discussed, in equilibrium, the price process is a path-dependent functional of the order flow $Y$. This path-dependence can be in fact simplified by introducing a state variable that will allow us to state and solve the problem as 
a finite dimensional Markov process control problem. The following lemma defines a fictitious state process for all $\phi\in C_{l,\a}$. The relevant state variable is defined endogenously via a fixed point condition on $\phi$.

\begin{lemma}\label{lem:measurability}
For all $l \in(0,\frac{1}{T\sigma^2\gamma})$, $\phi\in C_{l,\a}$ and strategy of the informed trader, there exists a unique semimartingale $\xi=\xi^\phi$ satisfying $\P$-a.s. for all $t\in [0,T]$,
\begin{align}\label{eq:dynxi}
\chi(t,\xi_t)=\chi(0,0)+Y_t+\int_0^t {\gamma\sigma^2}P (r,\xi_r)dr.
\end{align}
Additionally, $Y$ and $\xi$ generate the same filtration and the dynamics of $\xi$ is given by 
\begin{align}\label{eq:dxi}
d\xi_t=\frac{1}{\pa_\xi\chi(t,\xi_t)}\left(dY_t-\frac{\pa_{\xi\xi}\chi(t,\xi_t)}{2(\pa_\xi\chi(t,\xi_t))^2}(d\langle Y\rangle_t- d\langle \sigma B\rangle_t)\right)
\end{align}
and the mapping $(t,y_\cdot)\in \Lambda\mapsto \xi_t=\xi_t(y_\cdot)$ is non-anticipative and is in $C^{1,2}(\Lambda)$. 
\end{lemma}
\begin{proof}
The proof of the Lemma is provided in Section \ref{s.construct}.
\end{proof}

{Before proceeding with the construction of the fixed point we provide some interpretation of the quantities introduced. If the function $P$ and $\chi$ are given, \eqref{eq:dynxi} allows us to construct a novel state variable $\xi$ that generates the same filtration as $Y$. This process has the same role as the process $\xi$ defined at \cite[Hypothesis 4]{c}. The process $\frac{1}{\pa_\xi \chi(t,\xi_t)}$ has the role of the price pressure in the sense of \cite{c} and the function  $\frac{1}{\pa_\xi \chi}$ is a weight function in the sense of \cite{ccd}. The correct function will be determined by the choice of an appropriate $\phi$ which is still to be determined. The function $\Gamma$ has the role of a convex potential function. With our choice of $\chi$ and $\Gamma$, the expected utility of the informed trader can be computed via a simple convex conjugation as in \eqref{eq:expansionutility}.

We now introduce a novel process allowing us to describe the equilibrium dynamics of $\xi$ under $\cF^m$. For all $l \in(0,\frac{1}{T\sigma^2\gamma})$ and $\phi\in C_{l}$, let $\xi^0=\xi^{0,\phi}$ (the superscript $0$ means $0$ order flow from the informed trader) be the solution of 
\begin{align}\label{eq:sde}
\xi^0_t=\int_0^t\frac{\sigma}{(1-\gamma\sigma^2 (T-r)\pa_\xi P (r,\xi^0_r))} dB_r=\int_0^t\frac{\sigma}{\pa_\xi \chi (r,\xi^0_r)} dB_r
\end{align}
(i.e. $Y=\sigma B$ in \eqref{eq:dynxi}) and $\mu(=\mu^\phi)$ the distribution of $\xi^0_T$. Note that we only need distributional properties of $\xi^0$. Thus, this process could have been defined in a different probability space. }

The following proposition provides the fundamental solution to the Fokker Planck equation associated to \eqref{eq:sde} which is needed to establish the fixed point.
\begin{prop}\label{prop:f}
For all $l \in(0,\frac{1}{T\sigma^2\gamma})$, $\phi\in C_{l,\a}$, and $(r,x)\in [0,T]\times \R$ with $r<T$,  
the solution of the equation 
$$ \xi^0_t=x+\int_r^t\frac{\sigma}{\pa_\xi \chi^\phi (s, \xi^0_s)} dB_s,\mbox{ for }t\in [r,T]$$
 has density $\P( \xi^0_t\in dy)=G^\phi(r,x,t,y)dy$ with
\begin{align}\label{eq:densityG}
G^\phi(r,x,t,y)=\pa_\xi\chi^\phi(t,y) \frac{\exp\left(\gamma \Gamma^\phi(t,y)-\gamma\Gamma^\phi(r,x)-\frac{|\chi^\phi(t,y)-\chi^\phi(r,x)|^2}{2\sigma^2(t-r)}\right)
}{\sqrt{2\pi \sigma^2 (t-r)}}.
\end{align}
In particular, 
$\mu^\phi$ has a density which is given by 
\begin{align}\label{eq:fphi}
f_\phi:y\in \R\mapsto \frac{1}{\sqrt{2\pi \sigma^2 T}}\exp\left(\gamma \phi(y)-\gamma \Gamma^\phi(0,0)-\frac{|\chi^\phi(0,0)-y|^2}{2\sigma^2T}\right).
\end{align}
and $f_\phi$ only depends on $\{\pa_{\xi\xi}\phi(\xi):\xi\in \R\}$ and not on the choice of antiderivative of $\pa_{\xi\xi}\phi$. 
\end{prop}
\begin{proof}
The proof is provided in Section \ref{s.construct}.
\end{proof}

Note that given $\phi\in C_{l,\a}$ (and therefore $P=P^\phi$ given by \eqref{eq:pdep}) we have defined $\chi=\chi^\phi$ via the equality \eqref{eq:defzg}. Given $\chi^\phi$ we have defined $\mu^\phi$ as the law of the marginal of the solution of \eqref{eq:sde} at time $T$. Given $\mu^\phi$ (in fact absolutely continuous with density \eqref{eq:fphi}), the Brenier theorem\footnote{We have provided in the Appendix \ref{s.optt} a summary of results from optimal transport theory that we need to construct the equilibrium. } in \cite{b,m} yields the existence and uniqueness (up to an additive constant) of a convex function that we denote $\Psi^\phi$ such that 
$\pa_\xi \Psi^\phi$ pushes $\mu^\phi$ onto $\nu$. The fixed point condition that we need to be able to construct an equilibrium is $\phi^*=\Psi^{\phi^*}$ up to an additive constant. Due to the condition at $0$ in the definition of $C_{l,\a}$, the additive constant to determine $\Psi^{\phi}$ is chosen to have $\Psi^\phi\in C_{l,\a}$. 
We now state the main result regarding the existence of this fixed point. 
\begin{thm}\label{thm:fixed}
Under Assumption \ref{assum:main} i, there exists $\gamma_0>0$, so that for all $\gamma\in (0,\gamma_0)$, there exists a  convex function $\phi^*\in C_{\frac{2}{\sqrt{\kappa\sigma^2 T}},\a}$ so that $\phi^*$ is the Brenier potential whose derivative $\pa_\xi \phi^*$ pushes $\mu^{\phi^*}$ to $\nu$ where $\mu^{\phi^*}$ is the distribution of $\xi^{0,\phi^*}$ defined by \eqref{eq:sde} and $P^*=P^{\phi^*}$ defined by \eqref{eq:pdep}.

Under Assumption \ref{assum:main} ii, the same result holds with $\phi^*\in C_{l,\a}$ for some $l$ depending on $\inf\{p_\nu(x):p_\nu(x)>0\}$ and the support of $\nu$. 
\end{thm}
\begin{proof}
The proof of the theorem is provided in Subsection \ref{ss:prooffixed}.
\end{proof}

We make the convention that quantities with superscript $*$ are associated to $\phi^*$. 

The fixed point condition can in fact be stated as a functional equality. In dimension one the Brenier map can be explicitly written via the cumulative functions. Since $\mu^\phi$ has density $f_\phi$ defined in \eqref{eq:fphi},  denote $F_\phi(x)=\int_{-\infty}^x f_\phi(y)dy$. Then, as stated in Appendix \ref{s.optt}, the fixed point condition is to find $\phi^*$ satisfying for $\xi\in \R$, 
\begin{align}\label{eq:opttcon}
\pa_\xi \phi^*(\xi)=F^{-1}_\nu(F_{\phi^*}(\xi)).
\end{align}

We also note that one can state this fixed point condition in multidimension. 
In this case, the Brenier map is not explicit an one needs to solve the Monge-Ampere equation associated to transport of $ f_\phi$ to $p_\nu$, which is to find a convex function $\phi^*$ satisfying 
$$\det (\pa_{\xi\xi}\phi^*)=\frac{f_{\phi^*}}{p_\nu(\pa_\xi\phi^*)}.$$
We choose to exclude the multidimensional problem from our study to avoid increasing the length of the paper. However, we expect our methods to carry over to such multidimensional cases. 

 \section{Main Results}\label{s.mr}
We now use the fixed point $\phi^*$ constructed in Theorem \ref{thm:fixed} to establish an equilibrium. Recall that $G^{\phi^*}$ denotes the transition density of $\xi^{0,*}:=\xi^{0,\phi^*}$, $(\pa_\xi \phi^*)^{-1}$ denotes the inverse mapping of $ \pa_\xi\phi^*$ that is defined on the support of $\nu$, and define the function 
\begin{align}\label{eq:insider}
\theta^{\tilde v}(t,\xi)&=\frac{\sigma^2}{\pa_\xi\chi (t,\xi)}\frac{\pa_\xi G(t,\xi,T,( \pa_\xi\phi^*)^{-1}(\tilde v))}{G(t,\xi,T,( \pa_\xi\phi^*)^{-1}(\tilde v))}\notag\\
&=-\gamma\sigma^2 P(t,\xi)-\frac{\chi(t,\xi)-( \pa_\xi\phi^*)^{-1}(\tilde v)}{T-t}=\frac{( \pa_\xi\phi^*)^{-1}(\tilde v)-\xi}{T-t}
\end{align}
where we have used \eqref{eq:defzg} and \eqref{eq:densityG} to simplify the expression. If the informed trader uses the strategy $dX^*_t=\theta^{\tilde v}(t,\xi_t)dt$, 
then, 
\begin{align}\label{eq:dxi2}
d\xi_t=\frac{\sigma^2}{(\pa_\xi\chi (t,\xi_t))^2}\frac{\pa_\xi G(t,\xi_t,T,( \pa_\xi\phi^*)^{-1}(\tilde v))}{G(t,\xi_t,T,( \pa_\xi\phi^*)^{-1}(\tilde v))}dt+\frac{\sigma}{\pa_\xi\chi(t,\xi_t)} dB_t.
\end{align}
Thus, $\theta^{\tilde v}$ is the usual Doob's h-transform drift (of Y) to move $\xi$ to $P^{-1}(T,\tilde v)=(\pa_\xi\phi^*)^{-1}(\tilde v).$
Additionally, thanks to the Lemma \ref{lem:measurability}, $\xi$ is a $C^{1,2}(\Lambda)$ functional of the paths of $Y$. 
Denote by $H^*$ the mapping
\begin{align}\label{def:ps} 
(t,y_\cdot)\in \Lambda \mapsto H^*(t,y_\cdot):= P^{*}(t,\xi_t(y_\cdot)):=P^{\phi^*}(t,\xi_t(y_\cdot))
\end{align}
where $P^{\phi^*}$ is the solution of \eqref{eq:pdep} with final condition $\pa_\xi\phi^*$.  
We are now ready to state the main theorem regarding the existence of equilibrium. 
\begin{thm}\label{thm:eq}
{Under Assumption \ref{assum:main}, there exists $\gamma_1\in (0,\gamma_0]$ where $\gamma_0$ is as in Theorem \ref{thm:fixed} so that for all  $\gamma \in (0,\gamma_1)$, the pricing rule 
\begin{align}\label{eq:eqprice}
(t,y_\cdot)\in \Lambda \mapsto H^{*}(t,y_\cdot)
\end{align} for the market maker and 
the trading strategy 
\begin{align}\label{eq:opt}
dX^*_t= \theta^{\tilde v}(t,\xi_t)dt
\end{align}
for the informed trader form an equilibrium.

Additionally, denoting $\phi^c$ is the convex conjugate of $\phi^*$, at this equilibrium, the expected utility of the informed trader is 
\begin{align}\label{eq:utii}
-e^{\frac{\tilde v^2\gamma^2\sigma^2T}{2}+\gamma(\tilde v \chi^*(0,0)-\Gamma^*(0,0))-\gamma\phi^c(\tilde v)}
\end{align}
and the $\cF^m$ distribution of $\xi$ is the distribution of $\xi^0$. 
}
\end{thm}
\begin{proof}
The proof of the main theorem will be provided in Sections \ref{s:insider} and \ref{s:mm} by showing the rationality of the price \eqref{eq:eqprice} and optimality of the strategy \eqref{eq:opt} separately. 
\end{proof}

\subsection{Properties of the equilibrium}\label{ss.prop}

Similarly to \cite{kyle,ba,c}, we denote $\lambda(t,\xi)$, the price impact or Kyle's Lambda which is the sensitivity of the increments of the price in the increments of $Y$.
Due to \eqref{eq:pdep}, if the strategy of the informed trader is absolutely continuous, we have that 
$$dP^*(t,\xi_t)=\pa_\xi P^*(t,\xi_t)d\xi_t=\frac{\pa_\xi P^*(t,\xi_t)}{1-\gamma\sigma^2 (T-t)\pa_\xi P^*(t,\xi_t)}dY_t.$$
Therefore, the price impact has the expression, $\lambda(t,\xi)=\frac{\pa_\xi P^*(t,\xi)}{1-\gamma\sigma^2 (T-t)\pa_\xi P^*(t,\xi)}$. 
In \cite{c}, the author also obtains an expression for this price impact when $\nu$ is Gaussian and using their notations $\pa_\xi P^*$ is a constant denoted $\lambda^*(1)$ in \cite[Proposition 3]{c}. 
In our framework $\pa_\xi P^*$ is not constant anymore and in fact in the Appendix we show that, for any absolutely continuous strategy of the informed trader, the price impact has the dynamics
\begin{align}
d\lambda(t,\xi_t)&=-\gamma\sigma^2 \lambda^2(t,\xi_t) dt + \frac{\pa_{\xi\xi} P^*(t,\xi_t)}{(\pa_\xi\chi)^3(t,\xi_t)} dY_t
\end{align}
{Since $\xi$ is a $\cF^m$-martingale, the $\cF^m_t$ conditional expectation of the drift in \eqref{eq:dxi2} is $0$ and $Y$ is a martingale under the same filtration}. Thus, the price impact is a supermartingale. The intuition behind this is the fact that close to maturity the market maker expects that the informed trader has presumably large position. Therefore, being risk averse she will be less inclined to take larger positions and will be using less of her informational advantage close to maturity. Thus, there is less adverse selection risk towards the end, hence the market makers do not need to collect excessive rents to compensate. 

Additionally, the market depth $\zeta(t,\xi_t)=\frac{1}{\lambda(t,\xi_t)}$ satisfies 
\begin{align}\label{eq:depthd}
d\zeta(t,\xi_t)=&\bigg[\gamma\sigma^2+\frac{1}{(\pa_\xi P^*)^2(t,\xi_t)}.\frac{\gamma\sigma^4(T-t)(\pa_{\xi\xi}P^*)^2(t,\xi_t)}{(\pa_\xi \chi)^3(t,\xi_t)}+\frac{\sigma^2(\pa_{\xi\xi}P^*)^2(t,\xi_t)}{(\pa_\xi\chi)^2(t,\xi_t)(\pa_\xi P^*)^3(t,\xi_t)}\bigg]dt\notag\\ -&\Big[\frac{\pa_{\xi\xi}P^*(t,\xi_t)}{(\pa_\xi P^*)^2(t,\xi_t)\pa_\xi\chi (t,\xi_t)}\Big] dY_t
\end{align}
and is a submartingale. The property means that there will be more liquidity close to maturity.

\subsection{Comparison with \cite{c}}\label{s.cho}
In \cite{c}, the author also studies the Kyle-Back model where $\nu$ is Gaussian. In this case, relying on the linear quadratic structure, we can explicitly solve \eqref{eq:pdep}. Using the notations of \cite{c}(up to changing the sign of $\gamma$ and $T=1$), we conjecture the form 
$P^*(t,\xi)=\lambda \xi+m$ for the solution of \eqref{eq:pdep} for some $\lambda$ and $m$ to be determined.
The fixed point condition in particular imposes that 
$$var\left(\lambda \sigma \int_0^1 \frac{dB_t}{ (1-\gamma\sigma^2\lambda(1-t))}\right)=var(\tilde v)=\Sigma^2\mbox{ and }m=\E\tilde v.$$
We can solve for $\lambda$ and obtain $\lambda=-\frac{\gamma\Sigma^2}{2}+\sqrt{\frac{\gamma^2\Sigma^4}{4}+\frac{\Sigma^2}{\sigma^2}}$ which is $\lambda^*(1)$ mentioned in Subsection \ref{ss.prop}. 
Thus, unsurprisingly, we find the equilibrium in \cite{c} using our fixed point condition in Theorem \ref{thm:fixed}.

We can also compute $\Gamma$ and $\chi$ as
$$\Gamma(t,\xi)=\lambda (1-\gamma\sigma^2\lambda(1-t))\frac{\xi^2}{2}+f_t,\mbox{ and }\chi(t,\xi)= (1-\gamma\sigma^2\lambda(1-t))\xi$$
where $f'(t)=-\frac{\sigma^2\lambda}{2( 1-\gamma\sigma^2\lambda(1-t))}$.

Injecting in \eqref{eq:expansionutility}, we obtain we obtain the expected utility of the informed trader 
$$\sup_X\E\left[-\exp(-\gamma W_T)\right]= -\exp\left({\frac{\tilde v^2\gamma^2\sigma^2T}{2}+\gamma(\tilde v \chi(0,0)-\Gamma(0,0))-\gamma\frac{(\tilde v-m)^2}{2\lambda}}\right).$$
In this case, the dynamics of the price is given by 
$$dP^*(t,\xi_t)=\pa_\xi P^*(t,\xi_t)d\xi_t=\frac{\lambda }{1-\gamma\sigma^2 (1-t)\lambda}dY_t.$$
Thus, the market depth satisfies 
$$d\left(\frac{ 1-\gamma\sigma^2 (1-t)\lambda}{\lambda}\right)=\gamma \sigma^2dt.$$
Note that the growth rate of market depth in the general case of \eqref{eq:depthd}
is always greater or equal than the Gaussian case studied in \cite{c}. In this sense, the stochasticity of the price impact introduced due to the lack of Gaussianity yields to a larger market depth close to maturity. 
\subsection{Risk neutral case}
In this case, trivially, $\chi(t,\xi)=\xi$ and \eqref{eq:system} becomes the classical pricing PDE whose final condition is the Brenier map pushing $\mu^\phi$ onto $\nu$. Note that in this case, $\xi_t=Y_t$ for all $t\in [0,T]$ and $\mu^\phi$ does not depend on $\phi$ and is the Gaussian distribution of $\sigma B_T$. Thus, the fixed point is trivial and the Brenier map in question is just $\pa_\xi\phi(x)=F_\nu^{-1}(N(\frac{x}{\sigma}))$. Then, the equation \eqref{eq:pdep} simplifies to the heat equation. We obtain the equilibrium in the seminal papers \cite{kyle,ba}. \cite{back1993,cel,bou} studied the same problem with multidimensional $Z$ and even non absolutely continuous $\nu$. 
\subsection{The structure of path-dependence}
{In \cite{cn2}, the authors also study the Kyle-Back model and provide necessary and sufficient conditions on path-dependent functionals for the existence of equilibrium. Then, they use these results in \cite{cn} to establish the existence of equilibrium in some cases. Rather than giving necessary and sufficient conditions on its existence, we directly construct an equilibrium. One can then check that the equilibrium constructed here satisfies the assumptions in \cite{cn2,cn} and our pricing rules satisfy the necessary conditions provided in \cite{cn}. Indeed, using the notations of \cite{cn} (up to changing the sign of the risk aversion parameter),  the authors show that if the path-dependence of the pricing rule on $Y$ has the form 
$$P_t=H(t,\xi_t)\mbox{ with }\xi_t=\int_0^t \lambda(s,P_s)dY_s,$$
then, $H$ and $\lambda$ have to solve
\begin{align}
 \pa_tH+\frac{\sigma^2\lambda^2\pa_{\xi\xi}H}{2}&=0\label{blabla1}\\
 \frac{\pa_t\lambda}{\lambda^2}+\frac{\sigma^2(\pa_{\xi}H)^2\pa_{pp}\lambda}{2}&=-\gamma \sigma^2 \pa_{\xi}H\label{blabla2}.
\end{align}
Note that in our formulation, $H(t,\xi)=P^*(t,\xi)$. 
Additionally, by \eqref{eq:system2}, we have 
$$\pa_\xi\left (-\pa_t \chi+\frac{\sigma^2}{2}\pa_\xi \frac{1}{\pa_\xi \chi}\right)=-\gamma\sigma^2 \pa_\xi P$$
which implies that 
$$\frac{\pa_t \frac{1}{\pa_{\xi} \chi}}{\frac{1}{(\pa_{\xi} \chi)^2}}+\frac{\sigma^2}{2}\pa_{\xi\xi} \frac{1}{\pa_\xi \chi}=-\gamma\sigma^2 \pa_\xi P.$$
We can now use this equality combined with \eqref{eq:pdep} to show that the function 
$\lambda(t,p)=\frac{1}{\pa_\xi \chi(t,(P^*)^{-1}(t,p))}$
satisfies \eqref{blabla2} and in the equilibrium we construct, the structure of path-dependence of the pricing rule in $Y$ satisfies the assumptions in \cite{cn}.

Note that our main contribution is not to exhibit particular solutions to the system \eqref{blabla1}-\eqref{blabla2} but to use optimal transport theory to find the correct final condition $P^*(T,\cdot)$ and pinpoint appropriate solutions of this system so that the pricing rule of the market maker and the strategy of the informed trader satisfy additional conditions stated in \cite{cn} such as  $P_T=\tilde v$ and $\cF^m$-martingality of $Y$. 
}

\subsection{Numerical method to compute the equilibrium}\label{s.numerical}
In this subsection, we describe a numerical method we use to compute an approximate fixed point. Although in our numerical experiments the algorithm converges, the proof of this convergence is out of the scope of this paper. 

Since the Brenier map in one dimension is explicit, the fixed point condition on $\phi^*$ is
\begin{align}\label{eq:fixcond2}
\pa_\xi\phi^*(\xi)=F_\nu^{-1}(F_{\phi^*} (\xi)).
\end{align}

We start a fixed point iteration with $\pa_\xi\phi^0=0$ and we truncate $\pa_\xi\phi^*$ for large values of $\xi$.  

Assuming $\pa_\xi\phi^{n}$ is defined, thanks to Proposition \ref{prop:f}, we can compute 
$f_{\phi^n}$ and therefore $F_{\phi^n}$. For this computation we need to compute $\Gamma^{\phi^n}$ and $\chi^{\phi^n}$ which can be done via an accelerated gradient descent algorithm as described in Lemma \ref{lem:rep}. Given $F_{\phi^n}$, we now use the identity \eqref{eq:fixcond2} to define $\pa_\xi\phi^{n+1}$ via 
\begin{align}\label{eq:fixcond3}
\pa_\xi\phi^{n+1} (\xi)=F_\nu^{-1}(F_{\phi^{n}} (\xi)).
\end{align}
 We then iterate over $n$. This algorithm converges for the examples presented below and the limit is $\pa_\xi\phi^*$ where $\phi^*$ is a fixed point whose existence is proven in Theorem \ref{thm:fixed}. 
 
 Given the value of $\phi^*$, we can solve both the backward equations \eqref{eq:system}-\eqref{eq:system2} to obtain $\Gamma^{*}:=\Gamma^{\phi^*}$ and $\chi^*:=\chi^{\phi^*}$
 and the Fokker-Planck equation associated to $\xi^{0,*}=\xi^{0,\phi^*}$ to obtain $G$ in Proposition \ref{prop:f} as 
 \begin{align*}
G^*(t,\xi,T,y):=G^{\phi^*}(t,\xi,T,y)=\frac{\exp\left(\gamma \phi^*(y)-\gamma\Gamma^*(t,\xi)-\frac{|y-\chi^*(t,\xi)|^2}{2\sigma^2(T-t)}\right).
}{\sqrt{2\pi \sigma^2 (T-t)}}.
\end{align*}
Given the fact that $P(T,\xi_T)=\pa_\xi\phi^*(\xi_T)$, the distribution of $\tilde v$ conditional to the state $t\in (0,T)$ and $\xi_t=\xi$ is the pushforward of $G^*(t,\xi,T,y)dy$ via the map $\pa_\xi\phi^*$. This pushforward measure admits the density
\begin{align}\label{eq:densityprice}
G^*(t,\xi,T,(\pa_\xi\phi^*)^{-1}(y))\pa_\xi((\pa_\xi\phi^*)^{-1})(y)\1_{supp (\nu)}(y).
\end{align}

\subsubsection{Numerical results: $\nu$ Normal }
As discussed in Subsection \ref{s.cho}, if $\nu$ is normal, $\pa_\xi\phi$ is a linear function and therefore $\phi$ is quadratic. Since we have the same fixed point as in \cite{c}, the fixed point iteration in Subsection \ref{s.numerical} converges to a linear function $P^*(T,\xi)=\pa_\xi\phi^*(\xi)=\lambda^*(1)\xi$ described in \cite{c}(see Figure \ref{fig:1}).

\begin{figure}[hbt!]
\includegraphics[width=0.7\textwidth]{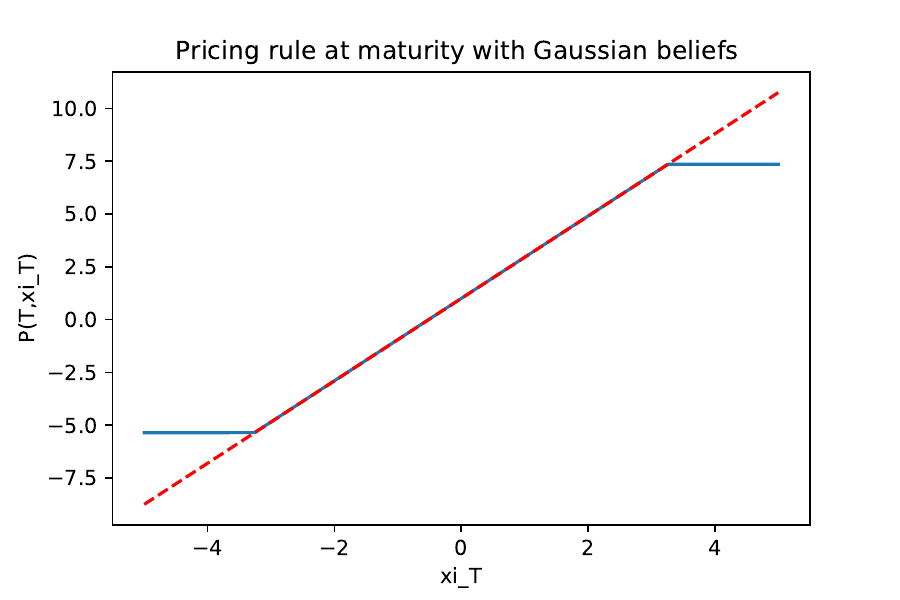}
\caption{The blue curve (truncated for large values of $\xi_T$) is computed with the fixed point iteration described in Subsection \ref{s.numerical} with parameters $(T,\sigma,\gamma)=(1,.5,.1)$ and $\nu\sim N(1,1)$. The red curve is obtained by the explicit formula of \cite{c}}
\label{fig:1}
\end{figure}

\subsubsection{Numerical results: $\nu$ uniform and {log-normal}}
In Figure \ref{fig:test}, we have the results of fixed point iteration for uniform and {log-normal} $\nu$.  As mentioned the standing Assumption \eqref{assum:main} is not satisfied for the {log-normal} distribution. However, our algorithm still converges for this distribution. The convergence of the numerical method suggests that the existence of the fixed point can be established for more general cases, but this point is out of the scope of our current work.
\begin{figure}[hbt!]
\centering
\begin{subfigure}{.5\textwidth}
  \centering
\includegraphics[width=1\textwidth]{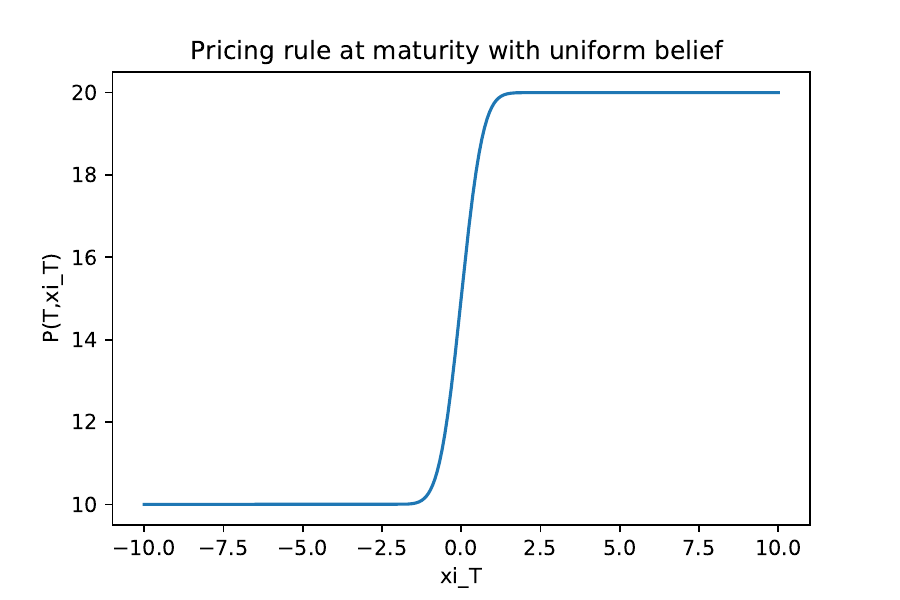}
  \caption{ $\nu\sim Unif(10,20)$}
  \label{fig:sub1}
\end{subfigure}%
\begin{subfigure}{.5\textwidth}
  \centering
\includegraphics[width=1\textwidth]{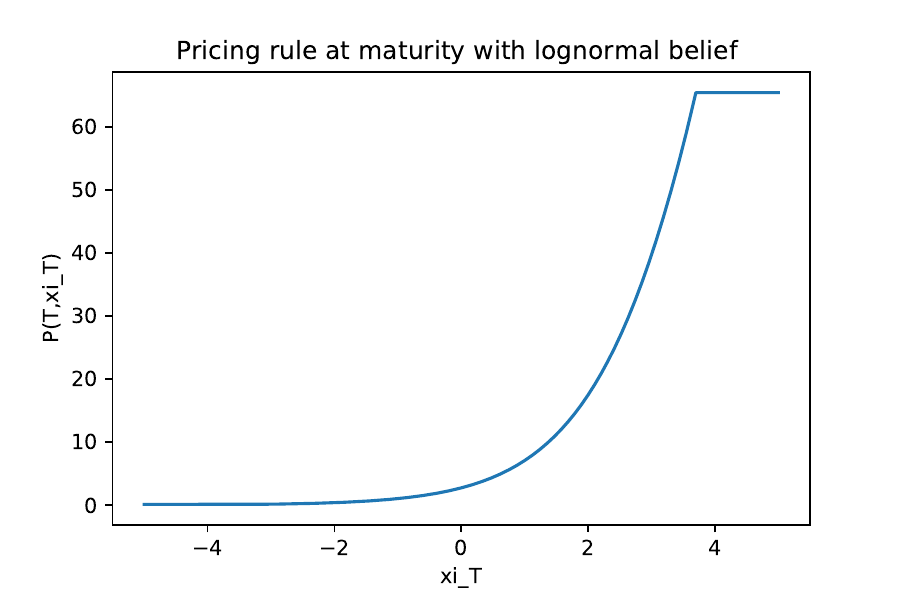}
  \caption{$\nu\sim exp(N(1,.5))$}
  \label{fig:sub2}
\end{subfigure}
\caption{Numerical computation of the fixed point for $(T,\sigma,\gamma)=(1,.5,.1)$}
\label{fig:test}
\end{figure}
\begin{figure}[hbt!]
\centering
\begin{subfigure}{.5\textwidth}
  \centering
\includegraphics[width=7.5cm, height=4.8cm]{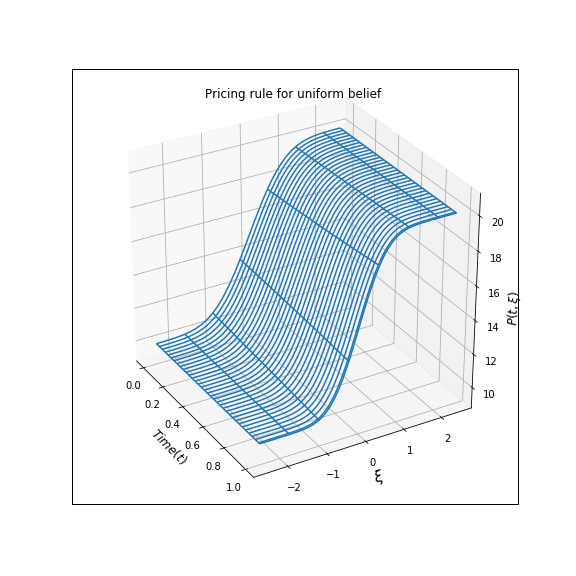}
  \label{fig:sub21}
\end{subfigure}%
\begin{subfigure}{.5\textwidth}
  \centering
\includegraphics[width=1\textwidth]{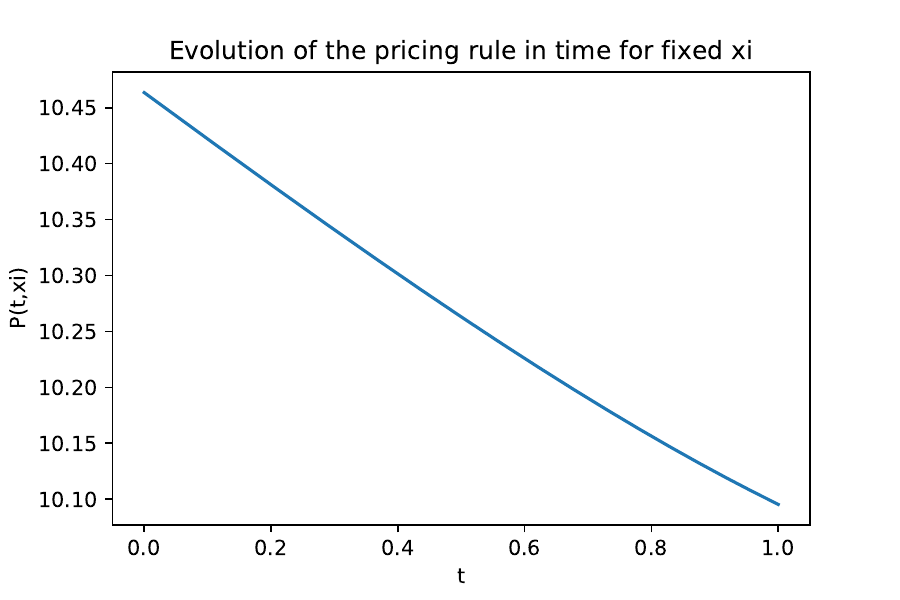}
  \label{fig:sub22}
\end{subfigure}
\caption{Pricing rule of the market maker (as a function of time and $\xi$ on the left and as a function of time for a fixed $\xi$ on the right) with parameters $(T,\sigma,\gamma)=(1,.5,.1)$ and $\nu\sim$Uniform$(10,20)$.}
\label{fig:test22}
\end{figure}

At each time $t\in (0,T)$ we can also compute the conditional distribution of $\tilde v$ given $\cF^m_t$ via \eqref{eq:densityprice}. 
As it should be, the expression \eqref{eq:densityprice}, interpolates the initial belief $\nu\sim$Unif$(10,20)$ of the market maker at $t=0$ to a Dirac mass (depending on the value of $\xi_T$) at $t=0.95\sim 1$. 

\begin{figure}[H]\label{fig:00}
{\caption{CDF of $\tilde v$ conditional to $(t,\xi_t)\in\{(0,0),(0.5,0),(0.95,0)\}$ versus the initial CDF of $\tilde v\sim$Unif$(10,20)$ with market parameters $(T,\sigma,\gamma)=(1,.5,.1)$.}}
{\includegraphics[width=0.5\textwidth]{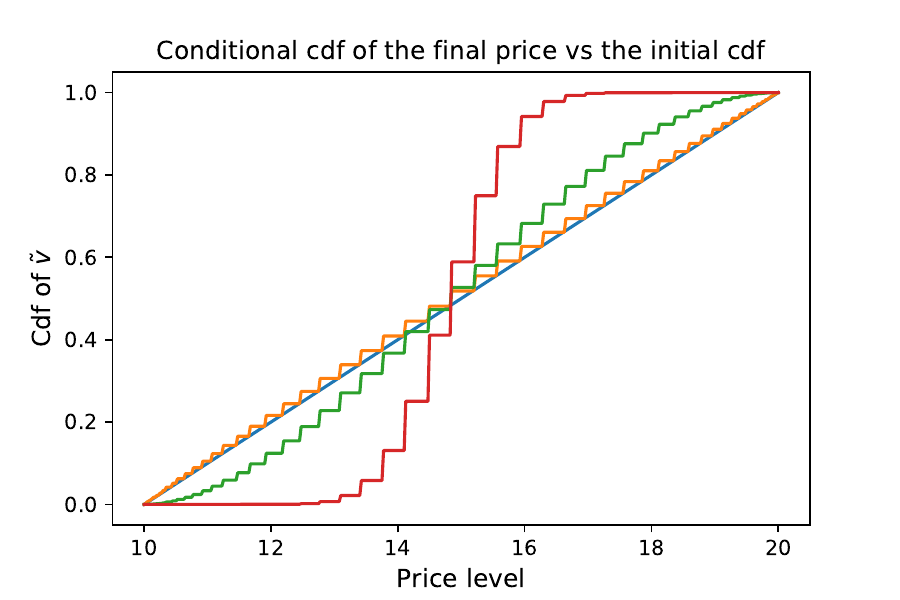}}
\end{figure}

The conditioning being on $\xi_t=0$ the Dirac mass is centered at $0$. However, in Figure \ref{fig:test222}, for different values of $\xi_t$(i.e. different aggregate order history) the Dirac mass moves on the support of $\nu$. 

\begin{figure}[hbt!]
\centering
\begin{subfigure}{.5\textwidth}
  \centering
\includegraphics[width=1\textwidth]{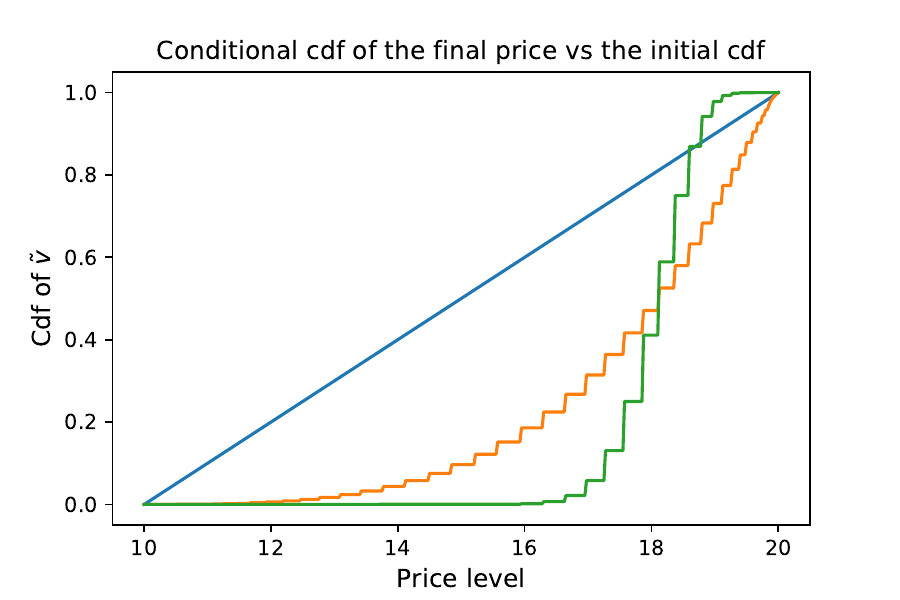}
  \label{fig:sub212}
  \caption{$(t,\xi_t)\in\{(0.5,0.5),(0.95,0.5)\}$. }
\end{subfigure}%
\begin{subfigure}{.5\textwidth}
  \centering
\includegraphics[width=1\textwidth]{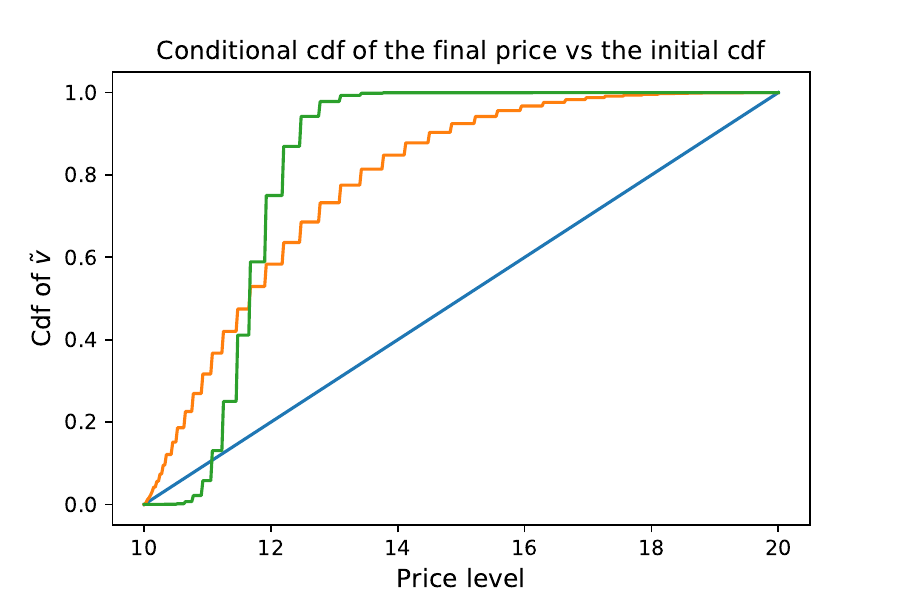}
  \label{fig:sub222}
  \caption{$(t,\xi_t)\in\{(0.5,-0.5),(0.95,-0.5)\}$. }
\end{subfigure}
  \caption{CDF of $\tilde v$ conditional to values of $(t,\xi_t)$ versus the initial CDF of $\tilde v\sim$Unif$(10,20)$ with market parameters $(T,\sigma,\gamma)=(1,.5,.1)$. }
\label{fig:test222}
\end{figure}

\section{Proof of Theorem \ref{thm:eq}}\label{s.proof1}
{Before proving the rationality of the pricing rule and the optimality of \eqref{eq:opt}, we first prove that $H^*\in \cH$. By definition $H^*(t,\sigma B_\cdot)$ has the same distribution as $P^*(t,\xi^{0,*}_t)$. The equations \eqref{eq:pdep} and \eqref{eq:sde} implies that $P^*(t,\xi^{0,*}_t)$ is a local martingale. Additionally, the equality $d P^*(t,\xi^{0,*}_t)= \frac{\pa_\xi P^*(t,\xi^{0,*}_t)}{\pa_\xi \chi^*(t,\xi^{0,*}_t)}\sigma dB_t$ and the bound \eqref{eq:lip} imply that the quadratic variation of $P^*(t,\xi^{0,*}_t)$ is uniformly bounded and therefore this process is a square integrable martingale and \eqref{eq:admm} holds.

Thanks to Lemma \eqref{lem:measurability}, $(t,y)\in \Lambda \mapsto\xi(t,y_\cdot)$ is in $C^{1,2}(\Lambda)$ with $\pa_y \xi(t,y_\cdot)=\frac{1}{\pa_\xi \chi^*(t,\xi_t)}$. Thus, by composition with smooth functions, $(t,y)\mapsto H^*(t,y_\cdot)=P^*(t,\xi(t,y_\cdot))$ is $C^{1,2}(\Lambda)$ and thanks to \eqref{eq:pathder}
its path derivative is 
$$\pa_yH^*(t,y_\cdot)=\frac{\pa_\xi P^*(t,\xi(t,y_\cdot))}{\pa_\xi\chi^*(t,\xi(t,y_\cdot))}>0$$
and $H^*\in \cH$. 
}
\subsection{Problem of the market maker}\label{s:mm}
We now show that if the informed trader uses the strategy \eqref{eq:opt}, then the pricing rule \eqref{eq:eqprice} is rational. 
Thanks to the regularity of $P$ established in Lemma \ref{lem:pdep}, the diffusion coefficient of $\xi$ is uniformly bounded away from $0$. Additionally, due to our fixed point condition, the distribution of $( \pa_\xi\phi^*)^{-1}(\tilde v)$ is the same as the distribution of $\xi_T^0$ and the generator of $\xi$ in \eqref{eq:dxi} is uniformly elliptic thanks to \eqref{eq:lip}. Thus, the results in \cite{ccd} insures that if the informed trader uses the strategy  \eqref{eq:opt} then, conditional to their filtrations, $\xi$ and $\xi^{0,*}$ have the same distribution and $\xi_T=( \pa_\xi\phi^*)^{-1}(\tilde v)$. The latter equality implies that $P_T=P^*(T,\xi_T)=\tilde v$. 
Thus, to obtain the rationality of $H^*$ it is sufficient to show that 
$P^*(t,\xi_t^{0,*})=\frac{\pa_{\xi}\Gamma(t,\xi_t^{0,*})}{\pa_{\xi}\chi(t,\xi_t^{0,*})}$ is a martingale. 
As mentioned above, the martingality of $P^*(t,\xi_t^{0,*})$ is a direct consequence of \eqref{eq:sde}, \eqref{eq:pdep} and the Ito's formula.

Therefore, if the informed trader uses the candidate equilibrium strategy, conditional to the filtration of $\xi$ and therefore to the filtration of $Y$, $P^*(t,\xi_t)$ is a martingale satisfying $P^*(T,\xi_T)=\tilde v.$

\subsection{Informed trader's problem}\label{s:insider}

We now show the martingality condition \eqref{eq:adi} by choosing $\gamma_1\in (0,\gamma_0]$ small enough allowing us to obtain the Novikov's condition 
\begin{align}\label{eq:nov}
\E\left[e^{\frac{\gamma^2\sigma^2}{2}\int_0^T(\tilde v- H^*(t,X^*_\cdot+Z_\cdot))^2dt}\right]<\infty
\end{align}
for all $\gamma\in (0,\gamma_1)$.
Thanks to the previous subsection and Jensen's inequality for all $n\geq 1$, we have 
\begin{align*}
\left(\frac{\gamma^2\sigma^2}{2}\int_0^T(\tilde v- H^*(t,X^*_\cdot+Z_\cdot))^2dt\right)^n&\leq \frac{T^{n-1}\gamma^{2n}\sigma^{2n}}{2^n}\int_0^T(\tilde v- H^*(t,X^*_\cdot+Z_\cdot))^{2n}dt\\
&\leq {T^{n-1}2^{n-1}\gamma^{2n}\sigma^{2n}}\int_0^T|\tilde v|^{2n}+|\E[\tilde v|\cF^m_t]|^{2n}dt\\
&\leq \frac{(2T\gamma^{2}\sigma^{2})^n}{2T}\int_0^T|v|^{2n}+\E[|\tilde v|^{2n}|\cF^m_t]dt
\end{align*}
Thus, 
we obtain $\E\left[\left(\frac{\gamma^2\sigma^2}{2}\int_0^T(\tilde v- H^*(t,X^*_\cdot+Z_\cdot))^2dt\right)^n\right]\leq  {(2T\gamma^{2}\sigma^{2})^n}\E[|\tilde v|^{2n}]$.

If $\nu$ has bounded support then the series  
$\sum_n \frac{1}{n!}{(2T\gamma^{2}\sigma^{2})^n}\E[|\tilde v|^{2n}]$ is convergent and we obtain \eqref{eq:nov} for $\gamma_1=\gamma_0$. 

If $\nu$ is strongly log concave as in Assumption \ref{assum:main} i) then, $p_\nu (x)\leq C \exp(-\frac{x^2}{2\frac{2}{\kappa}})$ for some constant $C$ and  
$$\E[|\tilde v|^{2n}] \leq C \frac{2^n}{\kappa^{n}}1\times 3\times \ldots\times (2n-1)\leq  C \frac{n! 4^n}{\kappa^{n}}$$
where the right hand side is an upper bound for the $2n$-th moment of a Gaussian distribution with variance $\frac{2}{\kappa}$. 
Thus, 
$$\E\left[e^{\frac{\gamma^2\sigma^2}{2}\int_0^T(\tilde v- H^*(t,X^*_\cdot+Z_\cdot))^2dt}\right]\leq C \sum_{n=0}^\infty \frac{(8T\gamma^{2}\sigma^{2})^n}{\kappa^{n}}
$$
is finite for all $\gamma \in (0,\gamma_1)$ for a choice of $\gamma_1\in (0,\gamma_0]$ small enough and we obtain \eqref{eq:nov}. 

Hence under Assumption  \ref{assum:main} we can choose $\gamma_1>0$ so that for all $\gamma\in (0,\gamma_1)$, we have $X^*\in \cA(H^*)$.

We now show that the trading strategy \eqref{eq:opt} is an optimizer for the informed trader's problem. 
Recall
\begin{align}\label{eq:pricing}
P_t=P^*(t,\xi_t)=P^{\phi^*}(t,\xi_t)=\frac{\pa_\xi\Gamma^{*}(t,\xi_t)}{\pa_\xi\chi^{*}(t,\xi_t)}:=\frac{\pa_\xi\Gamma^{\phi^*}(t,\xi_t)}{\pa_\xi\chi^{\phi^*}(t,\xi_t)}.
\end{align}
We now compute the optimal strategy of the informed trader against this pricing rule. 
Let $X\in \cA(H^*)$ be an admissible strategy for the informed trader with semimartingale decomposition
$dX_t={dA_t}+\a_t dB_t$ {for $A_t$ with continuous and finite variation.} By \eqref{eq:dxi}, we have
$$d\xi_t=\frac{1}{\pa_\xi\chi^*(t,\xi_t)}\left({dA_t}+(\sigma+\a_t) dB_t-\frac{\pa_{\xi\xi}\chi^*(t,\xi_t)}{(\pa_\xi\chi^*(t,\xi_t))^2}(\sigma\a_t+\frac{\a^2_t}{2})dt\right).$$
Applying Ito's formula we have that 
\begin{align*}
-\gamma d(\tilde v \chi^*(t,\xi_t)-\Gamma^*(t,\xi_t))=&\left(- \tilde v\gamma^2\sigma^2\frac{\pa_\xi\Gamma^*(t,\xi_t)}{\pa_\xi\chi^(t,\xi_t)}+\frac{\gamma^2\sigma^2(\pa_\xi\Gamma^*(t,\xi_t))^2}{2(\pa_\xi\chi^*)^2(t,\xi_t)}\right)dt\\
&- \tilde v\gamma \frac{((\sigma+\a_t)^2-\sigma^2)\pa_{\xi\xi}\chi^*(t,\xi_t)}{2(\pa_\xi\chi^*)^2(t,\xi_t)}dt\\
&+\frac{\gamma((\sigma+\a_t)^2-\sigma^2)\pa_{\xi\xi}\Gamma^*(t,\xi_t)}{2(\pa_\xi\chi^*)^2(t,\xi_t)}dt\\
&+\gamma\frac{\pa_{\xi\xi}\chi^*(t,\xi_t)}{(\pa_\xi\chi^*(t,\xi_t))^2}\left(\sigma\a_t +\frac{\a_t^2}{2}\right)\left(\tilde v-P_t\right)dt\\
&-\gamma(\tilde v-P_t)( {dA_t}+(\sigma+\a_t) dB_t).
\end{align*}
{Using \eqref{eq::pp}, we can simplify this expression as
\begin{align*}
-\gamma d(\tilde v \chi^*(t,\xi_t)-\Gamma^*(t,\xi_t))=&\frac{\gamma^2\sigma^2}{2}\left({(P_t-\tilde v)^2-\tilde v^2}\right)dt\\
&+\gamma\left(\sigma \a_t+\frac{\a_t^2}{2} \right)\frac{{\pa_{\xi}}P^*(t,\xi_t)}{{\pa_{\xi}}\chi^*(t,\xi_t)}dt\\
&-\gamma(\tilde v-P_t)( dX_t+\sigma dB_t).
\end{align*}
We also have
\begin{align*}
\gamma d\langle P,X\rangle_t&=\gamma(\sigma+\a_t)\a_t\frac{{\pa_{\xi}}P^*(t,\xi_t)}{{\pa_{\xi}}\chi^*(t,\xi_t)}dt.
\end{align*}
Therefore, we have the decomposition of the wealth as
\begin{align}\label{eq:dcmp1}
-\gamma W_T=&\int_0^T -\gamma(\tilde v-P_t)dX_t+\gamma \langle P,X\rangle_T\\
=&\gamma(\tilde v \chi^*(0,0)-\Gamma^*(0,0))-\gamma (\tilde v \xi_T-\phi^*(\xi_T))+\frac{\tilde v^2\gamma^2\sigma^2T}{2}\notag\\
&+\int_0^T\gamma(\tilde v-P_t)\sigma dB_t-\int_0^T\frac{\gamma^2\sigma^2}{2}(\tilde v-P_t)^2 dt+\int_0^T\frac{\gamma\a_t^2{\pa_{\xi}}P^*(t,\xi_t)}{2{\pa_{\xi}}\chi^*(t,\xi_t)}dt.\notag
\end{align}
Note that due to the condition $X\in \cA(H^*)$,
$$\exp\left(\int_0^\cdot \gamma \sigma (\tilde v-P_t)dB_t-\frac{\gamma^2\sigma^2}{2}\int_0^t (\tilde v-P_t)^2dt\right)$$
is a martingale and the utility of the informed trader is 
\begin{align}\label{eq:expansionutility}
&\E\left[-\exp\left(\int_0^T -\gamma(\tilde v-P_t) d X_t+\gamma \langle P,X\rangle_T\right)|\cF_0\right]\\
&=e^{\frac{\tilde v^2\gamma^2\sigma^2T}{2}+\gamma(\tilde v \chi^*(0,0)-\Gamma^*(0,0))}\tilde \E\left[-e^{-\gamma (\tilde v \xi_T-\phi^*(\xi_T))+\gamma\int_0^T\frac{\a_t^2{\pa_{\xi}}P^*(t,\xi_t)}{2{\pa_{\xi}}\chi^*(t,\xi_t)}dt}|\cF_0\right]\notag
\end{align}
where $\tilde \E$ is the equivalent probability measure obtained via Girsanov theorem under which 
$$dB_t-\gamma\sigma (\tilde v-P_t)dt$$ is a Brownian motion.

Define the convex conjugate of $\phi^*$, $\phi^c(v)=\sup_y vy-\phi^*(y)$. Thanks to \eqref{eq:expansionutility} and the positivity of ${\pa_{\xi}}P$ and ${\pa_{\xi}}\chi$, 
$$\E\left[-\exp(-\gamma W_T)|\cF_0\right]\leq -e^{\frac{\tilde v^2\gamma^2\sigma^2T}{2}+\gamma(\tilde v \chi^*(0,0)-\Gamma^*(0,0))-\gamma\phi^c(\tilde v)}$$
and any strategy of the informed trader that insures 
\begin{align}\label{eq:transport}
\tilde v={\pa_{\xi}}\phi^*(\xi_T)
\end{align} 
and $\a_t=0$
is a pointwise maximizer of the integrand under the expectation on the right hand side of \eqref{eq:expansionutility} and
is therefore optimal. 

Thanks to \cite[Theorem 2.1]{ccd}, if the informed trader uses the control \eqref{eq:opt} then, he insures
$\xi_T=( \pa_{\xi}\phi^*)^{-1}(\tilde v)$ and therefore \eqref{eq:opt} is optimal for the informed trader.
This concludes the proof of the theorem.

\begin{remark}
The positive term $\frac{{\pa_{\xi}}P^*(t,\xi_t)}{{\pa_{\xi}}\chi^*(t,\xi_t)}$ in \eqref{eq:expansionutility} renders any martingale part of the strategies of the informed trader suboptimal. Note that  $\frac{{\pa_{\xi}}P^*(t,\xi_t(y_\cdot))}{{\pa_{\xi}}\chi^*(t,\xi_t(y_\cdot))}=\pa_y H^*(t,y_\cdot)$. Thus, as proven in \cite{cn}, the positivity of $\pa_y H^*$ is the main penalization of martingale part of the strategies of the informed trader. 
\end{remark}}

\section{Construction of the Fixed Point} \label{s.construct}
We prove in this section the Theorem \ref{thm:fixed} and auxiliary results. 
\subsection{Proof of Lemma \ref{lem:pdep}}\label{ss:quasi}
We formally differentiate \eqref{eq:pdep} to obtain that $R=\textcolor{black}{\pa_{\xi}}P$ solves
\begin{align}\label{eq:pdepp}
 \textcolor{black}{\pa_{t}}R+\frac{\sigma^2\textcolor{black}{\pa_{\xi\xi}}R}{2(1-\gamma\sigma^2 R(t,\xi)(T-t))^2}+\frac{\gamma\sigma^4  \textcolor{black}{\pa_{\xi}}R^2(t,\xi)(T-t)}{(1-\gamma\sigma^2 R(t,\xi)(T-t))^3}&=0\\
R(T,\xi)&=\textcolor{black}{\pa_{\xi\xi}}\phi(\xi).
\end{align}
 We first establish wellposedness for this equation which will be needed to obtain global existence of the solutions of \eqref{eq:pdep}.
\begin{lemma}\label{lem:quasi}
Let $\phi\in C_{l,\a}$ with $l \in (0,\frac{1}{T\gamma \sigma^2})$. Then, there exists a unique $R=R^\phi:[0,T]\times \R\mapsto \R$ of class $C^{1,2}([0,T]\times \R)$ solving the PDE
\begin{align}\label{eq:pder}
\textcolor{black}{\pa_{t}} R+\frac{\sigma^2\textcolor{black}{\pa_{\xi\xi}}R}{2(1-\gamma\sigma^2 R(t,\xi)(T-t))^2}+\frac{\gamma\sigma^4  \textcolor{black}{\pa_{\xi}}R^2(t,\xi)(T-t)}{(1-\gamma\sigma^2 R(t,\xi)(T-t))^3}&=0\\
R(T,\xi)&=\textcolor{black}{\pa_{\xi\xi}}\phi(\xi).
\end{align}
which also satisfies $0\leq R\leq \sup\textcolor{black}{\pa_{\xi\xi}}\phi$. 
\end{lemma}
\begin{proof}

 This equation can be written as a quasilinear parabolic PDE of divergence form as 
 \begin{align}\label{eq:quasilin}
 \textcolor{black}{\pa_{t}} R+\pa_\xi\left(\frac{\sigma^2 \textcolor{black}{\pa_{\xi}}R}{2(1-\gamma\sigma^2 (T-t)R)^2}\right)&=0\\
  R(T,\xi)&=\textcolor{black}{\pa_{\xi\xi}}\phi(\xi).
 \end{align}
 The main difficulty is the fact that the denominator can become $0$. To avoid this technicality we penalize the PDE and show that this penalization still yields to a solution to \eqref{eq:quasilin}. 
 
 Pick a smooth monotone function with bounded derivatives 
 $$\Psi : \R\mapsto \R^+$$ 
 satisfying 
 $$\Psi(x)=x,\mbox{ if }x\geq 1-T\gamma\sigma^2 l\mbox{ and }\Psi(x)=\frac{1-T\gamma\sigma^2 l}{2} \mbox{ if }x\leq \frac{\e}{2}.$$
 
We now study the wellposedness for 
\begin{align}\label{eq:approximation}
   \textcolor{black}{\pa_{t}}\tilde R+\pa_\xi\left(\frac{\sigma^2 \textcolor{black}{\pa_{\xi}}\tilde R}{2\Psi^2(1-\gamma\sigma^2 \tilde R(t,\xi)(T-t))}\right)&=0\\
  \tilde R(T,\xi)&=\textcolor{black}{\pa_{\xi\xi}}\phi(\xi).
 \end{align}
 Due to the fact that $\textcolor{black}{\pa_{\xi\xi}}\phi\in C^{2+\a}_{loc}$ and bounded, \eqref{eq:approximation} satisfies all the assumptions of \cite[Chapter V, Theorem 6.1]{lsu}. Thus, there exists $\tilde R\in C^{1+\a/2,2+\a}_{loc}([0,T]\times \R)$ solving \eqref{eq:approximation}. 
  
We now show that $\tilde R\leq \sup \textcolor{black}{\pa_{\xi\xi}}\phi$.  
We know that $\tilde R$ is uniformly bounded. Denote $M=\sup \tilde R$. For $\e>0$ to be determined define 
$$\underline u(t,\xi)=\tilde R(t,\xi)-\frac{\e \xi^2}{2}$$ 
so that 
$\underline u$ solves 
\begin{align}\label{eq:approximation2}
    \textcolor{black}{\pa_{t}}\underline u+\pa_\xi\left(\frac{\sigma^2( \textcolor{black}{\pa_{\xi}}\underline u+\e\xi)}{2\Psi^2(1-\gamma\sigma^2 (T-t)(\underline u(t,\xi)+\frac{\e \xi^2}{2}))}\right)&=0\\
\underline u(T,\xi)&=\textcolor{black}{\pa_{\xi\xi}}\phi(\xi)-\frac{\e \xi^2}{2}.
 \end{align}
Additionally if $\xi^2= \frac{2M}{\e}$, $\underline u(t,\xi)\leq 0$. 

Define also the function $$\overline u(t,\xi)=\sup\textcolor{black}{(\pa_{\xi\xi}\phi)} e^{\e (\frac{\sigma^2}{2}+2\gamma\sigma^4T M)(T-t)}$$
which is a supersolution to  
\begin{align}\label{eq:approximation2}
    \textcolor{black}{\pa_{t}}\overline u+\pa_\xi\left(\frac{\sigma^2(\textcolor{black}{\pa_{\xi}} \overline u+\e\xi)}{2\Psi^2(1-\gamma\sigma^2 (T-t)(\overline u(t,\xi)+\frac{\e \xi^2}{2}))}\right)&\leq 0\mbox{ on }\xi^2\leq \frac{2M}{\e}\\
\overline u(T,\xi)&\geq \textcolor{black}{\pa_{\xi\xi}}\phi(\xi).
 \end{align}
Additionally, if $\xi^2= \frac{2M}{\e}$, $\overline u(t,\xi)\geq 0$. 
Thus, by the comparison result in \cite[Theorem 1.6]
{cr} $$\underline u(t,x)\leq \overline u(t,x). $$
Sending $\e\to 0$ we obtain that 
$$\tilde R\leq \sup\textcolor{black}{\pa_{\xi\xi}}\phi.$$
Similarly, we can also prove that 
$$0\leq \tilde R.$$
Thus, for all $(t,x)\in [0,T]\times R$, 
$$\Psi(1-\gamma\sigma^2 \tilde R(t,\xi)(T-t)))=1-\gamma\sigma^2 \tilde R(t,\xi)(T-t)$$
and 
$\tilde R=R$ solves \eqref{eq:pder}. The uniqueness comes from that fact that any solution of \eqref{eq:pder} satisfying $0\leq R\leq \sup\textcolor{black}{\pa_{\xi\xi}}\phi$ also satisfies \eqref{eq:approximation} which has a unique solution. 
\end{proof}

Given the Lemma \ref{lem:quasi} we can now prove the Lemma \ref{lem:pdep}.
\begin{proof}[Proof of Lemma \ref{lem:pdep}]
In Lemma \ref{lem:quasi}, we study $\textcolor{black}{\pa_{\xi}}P$ and show wellposedness for the equation solved by $R=\textcolor{black}{\pa_{\xi}}P$
\begin{align}\label{eq:pder2}
 \textcolor{black}{\pa_t} R+\frac{\sigma^2\textcolor{black}{\pa_{\xi\xi}}R}{2(1-\gamma\sigma^2 R(t,\xi)(T-t))^2}+\frac{\gamma\sigma^4  \textcolor{black}{\pa_{\xi}}R^2(t,\xi)(T-t)}{(1-\gamma\sigma^2 R(t,\xi)(T-t))^3}&=0\\
R(T,\xi)&=\textcolor{black}{\pa_{\xi\xi}}\phi(\xi).
\end{align}
Given the function $R$, one can show by a direct computation that 
\begin{align}\label{eq:constintp}
P(t,\xi)=\textcolor{black}{\pa_{\xi}}\phi(0)+\frac{\sigma^2}{2}\int_t^T\frac{\textcolor{black}{\pa_{\xi}}R (s,0)ds}{(1-\gamma\sigma^2 (T-s)R(s,0))^2}+\int_0^\xi R(t,r)dr
\end{align}
is a solution of \eqref{lem:pdep}. 

The uniqueness is a consequence of the fact that the derivatives of any two solution of \eqref{eq:pdep} has to solve \eqref{eq:pder2} which has a unique solution. Additionally, there is only one way of obtaining the constants of integration in $\xi$ as in \eqref{eq:constintp}.
\end{proof}

\begin{proof}[Proof of Proposition \ref{prop:f}]
Fix $l\in (0,\frac{1}{T\sigma^2\gamma})$ and $\phi\in C_{l,\a}$ and omit the dependence in these quantities. The inequality \eqref{eq:lip} implies that for all $t\in [0,T]$, $\xi\mapsto \chi(t,\xi)$ is invertible and we denote its inverse mapping who is $C^{1,2}([0,T]\times \R)$ as $\chi^{-1}$. We fix $r\in (0,1)$ and $x\in {\R}$, thanks to \eqref{eq:dynxi}, we have that the process $\chi_t=\chi(t,\tilde \xi_t)$ satisfies 
\begin{align*}\chi_t=\chi(r,x)+\sigma \tilde B_t+\int_r^t\sigma^2b (s,\chi_s)ds.
\end{align*}where $b(s,y)= {\gamma}P (s,\chi^{-1}(s,y))$ and $\tilde B_t=B_t-B_r$. 
Define ${J(s,y)}=\gamma \Gamma(s,\chi^{-1}(s,y))$. Then, 
\begin{align*}
{J(t,\chi(r,x)+\sigma\tilde B_t)}&=\gamma \Gamma(r,\chi^{-1}(r,\chi(r,x)))\\
&+\int_r^tb(s,\chi(r,x)+\sigma \tilde B_s)\sigma dB_s-\frac{\sigma^2}{2}\int_r^tb^2(s,\chi(r,x)+\sigma \tilde B_s)ds
\end{align*}
and therefore, 
\begin{align*}
e^{\int_r^tb(s,\chi(r,x)+\sigma {\tilde B_s})\sigma dB_s-\frac{\sigma^2}{2}\int_r^tb^2(s,\chi(r,x)+\sigma {\tilde B_s})ds}
=e^{\gamma \Gamma(t,\chi^{-1}(t,\chi(r,x)+\sigma \tilde B_t))-\gamma \Gamma(r,x)}.
\end{align*}
Thus,  \cite[Theorem 3.1]{tt} yields that $\chi_t$ has density 
\begin{align*}
y\mapsto \frac{\exp\left(\gamma \Gamma(t,\chi^{-1}(t,y))-\gamma\Gamma(r,x)-\frac{|y-\chi(r,x)|^2}{2\sigma^2(t-r)}\right).
}{\sqrt{2\pi \sigma^2 (t-r)}}
\end{align*}
 Therefore $\tilde \xi_t=\chi^{-1}(t,\chi_t)$ has density
$$G(r,x,t,y)=\chi_\xi(t,y) \frac{\exp\left(\gamma \Gamma(t,y)-\gamma\Gamma(r,x)-\frac{|\chi(t,y)-\chi(r,x)|^2}{2\sigma^2(t-r)}\right)
}{\sqrt{2\pi \sigma^2 (t-r)}}.$$
Combined with ${J(T,y)}=\gamma \Gamma(T,\chi^{-1}(T,y))=\gamma \phi(y)$ these densities yields that 
$\mu_\phi$ has density 
$$\frac{1}{\sqrt{2\pi \sigma^2T}}e^{-\frac{|y-\chi(0,0)|^2}{2\sigma^2T}}e^{\gamma \phi(y)-\gamma \Gamma(0,0)}=f_\phi(y).$$
Note also that for $l <\frac{1}{T\sigma^2\gamma}$ and $\phi\in C_{l,\a}$, we have that 
$$\gamma \phi(y)-\gamma \Gamma^\phi(0,0)-\frac{|\chi^\phi(0,0)-y|^2}{2\sigma^2T}\leq \frac{\gamma l}{2}(c_\phi+ (1-\frac{1}{{\gamma l \sigma^2 T}}) \frac{y^2}{2})$$
for a constant $c_\phi$ depending on $\phi$. 
Thus, $f_\phi$ is integrable (and of integral $1$). 

We now show that $f_{\phi}$ only depends on the second derivative of $\phi$. Recall $R$ defined in Lemma \ref{lem:quasi} which only depend on $\textcolor{black}{\pa_{\xi\xi}}\phi$. Using \eqref{eq:constintp} and \eqref{eq:constintg}, we have that 
\begin{align*}
P(0,0)&=\textcolor{black}{\pa_{\xi}}\phi(0)+A\\
\Gamma(0,0)&={\tilde A}-\frac{\gamma\sigma^2 T}{2}(\textcolor{black}{\pa_{\xi}}\phi(0)+A)^2
\end{align*}
where the terms
\begin{align*}
A&=\frac{\sigma^2}{2}\int_0^T\frac{\textcolor{black}{\pa_{\xi}}R (s,0)ds}{(1-\gamma\sigma^2 (T-s)R(s,0))^2}\\
{\tilde A}&=\int_0^T\frac{\sigma^2 R(s,0)}{2(1-\gamma \sigma^2(T-s)R(s,0))}ds
\end{align*}
only depend on $R$ and therefore on $\textcolor{black}{\pa_{\xi\xi}}\phi$. We now inject these expressions in \eqref{eq:fphi} to obtain thanks to \eqref{eq:defzg} that 
\begin{align*}
 &\phi(y)- \Gamma^\phi(0,0)-\frac{|\chi^\phi(0,0)-y|^2}{2\gamma\sigma^2T}\\
 &=\phi(y)-\textcolor{black}{\tilde A}+\frac{\gamma\sigma^2T}{2}(\textcolor{black}{\pa_{\xi}}\phi(0)+A)^2\\
 &-\frac{1}{2\gamma\sigma^2T}\left(y^2+2y\gamma\sigma^2T (\textcolor{black}{\pa_{\xi}}\phi(0)+A)+\gamma^2\sigma^4T^2 (\textcolor{black}{\pa_{\xi}}\phi(0)+A)^2\right)\\
 &=\phi(y)-y\textcolor{black}{\pa_{\xi}}\phi(0)-\textcolor{black}{\tilde A}-yA-\frac{y^2}{2\gamma\sigma^2T}.
\end{align*}
Since $\phi(0)=0$, using the identity 
$$\phi(y)-y\textcolor{black}{\pa_{\xi}}\phi(0)=\int_0^y\int_0^r\textcolor{black}{\pa_{\xi\xi}}\phi(s)dsdr$$
we have that 
\begin{align}\label{eq:fphi2}
f_\phi(y)= \frac{1}{\sqrt{2\pi \sigma^2 T}}\exp\left(\gamma\int_0^y\int_0^r\textcolor{black}{\pa_{\xi\xi}}\phi(s)dsdr-\gamma \textcolor{black}{\tilde A}-\gamma yA-\frac{y^2}{2\sigma^2T}\right)
\end{align}
which only depends on $\textcolor{black}{\pa_{\xi\xi}}\phi$. 
\end{proof}
Unfortunately, the stability estimates available in the literature for the solutions of \eqref{eq:pder2} or \eqref{eq:pdep} are not strong enough to establish the existence of fixed point $\phi^*$ (see \cite{cr,cmh,k,l}).  
We now provide a new stochastic representation for the solution of \eqref{eq:pdep} which is also of independent interest. To state and prove this representation we will be working with the derivative of $\phi$ instead of $\phi$ itself. For this purpose define 
\begin{align}\label{eq:defcx}
\cX_l=\left\{g\in Lip(\R):0 \leq g'\leq l \mbox{ a.e. }\right\}.
\end{align}
For $g\in \cX_l$, we denote $\tilde g$ its antiderivative that is $0$ at $0$.  Due to our definition, 
we have that for $g$ smooth enough $g\in \cX_l$ if and only if $\tilde g\in C_{l, \a}$. We now give the following lemma that studies the continuity of the map $g\in \cX_l\mapsto \chi^{\tilde g}(0,0)$. 

\begin{lemma}\label{lem:rep}
For $l <\frac{1}{2T\sigma^2\gamma}$, $g\in \cX_{l}$, and $z\in \R$, 
denote $$G_g(z)=\frac{1}{\sqrt{2\pi \sigma^2 T}}\int_\R e^{\frac{z^2}{2\sigma^2T}+\gamma \tilde g(z+y)-\frac{1}{2\sigma^2 T}y^2}dy=\E\left[e^{\frac{z^2}{2\sigma^2T}+\gamma \tilde g(z+\sigma B_T)}\right].$$ 
Then, by extending the domain of $\chi(0,0)$ and $\Gamma(0,0)$ (as function of $\tilde g$), we have $\chi^{\tilde g}(0,0)=\argmin_{z\in \R}G_g(z)$ and {$\Gamma^{ \tilde g}(0,0)=\min_{z\in \R}\frac{\ln G_g(z)}{\gamma}-\frac{(\chi^{{\tilde g}}(0,0))^2}{2\sigma^2\gamma T}$} for all $g\in \cX_l$.
Additionally, $\chi^{\tilde g}(0,0)$ and $\Gamma^{\tilde g}(0,0)$ are bounded by a constant that only depends on $(g (0),\gamma,\sigma, T)$ and 
 for $g^n\in \cX_{l}$ with 
\begin{align}\label{eq:condconv}
g^n\to g^0\mbox{ uniformly on compact sets},
\end{align} 
then, we have
\begin{align}\label{eq:convz}
(\chi^{\tilde g^n}(0,0),\Gamma^{\tilde g^n}(0,0))\to (\chi^{\tilde g^0}(0,0),\Gamma^{\tilde g^0}(0,0)).
\end{align}
\end{lemma}
\begin{proof}
Fix $l \in(0,\frac{1}{2T\sigma^2\gamma})$, and $g\in \cX_{l}$. We denote 
$\tilde \chi^{\tilde g}=\argmin_{z\in \R}G_{ g}(z)$ and $\tilde \Gamma^{\tilde g}=\min_{z\in \R}\frac{\ln G_g(z)}{\gamma}-\frac{(\chi^{{\tilde g}}(0,0))^2}{2\sigma^2\gamma T}$ if they exist.

{\it Step 1: Properties of $G_g$.} 
 We first show that $G_g$ is well defined and is twice continuously differentiable. 
Splitting the terms in the exponential, we have
\begin{align}\label{eq:intG}
G_g(z)=\E\left[e^{\frac{z^2}{2\sigma^2T}+\gamma {\tilde g}(z+\sigma B_T)}\right]&=\frac{ 1}{\sqrt{2\sigma^2T}}\int e^{\frac{z^2}{2\sigma^2T}+\gamma {\tilde g}(z+y)-\frac{3y^2}{8\sigma^2 T}} e^{-\frac{y^2}{8\sigma^2 T}}dy\\
\notag&=\frac{ 1}{\sqrt{2\sigma^2T}}\int e^{\frac{z^2}{2\sigma^2T}+\gamma {\tilde g}(z+y)+\frac{3y^2}{8\sigma^2 T}} e^{-\frac{7y^2}{8\sigma^2 T}}dy.
\end{align}
Since $0 \leq g'\leq l\leq \frac{1}{2T\sigma^2\gamma}$, if $|g(0)|\leq R$ and $|z|\leq R$ for some $R>0$, then, for all $y\in \R$, we have
\begin{align*}
\frac{z^2}{2\sigma^2T}+\gamma {\tilde g}(z+y)-\frac{3y^2}{8\sigma^2 T}&\leq \frac{z^2}{2\sigma^2T}+\frac{|z+y|^2}{4\sigma^2 T}+R\gamma|z+y|-\frac{3y^2}{8\sigma^2 T}\leq C_R
\end{align*}
and
{\begin{align*}
\frac{z^2}{2\sigma^2T}+\gamma {\tilde g}(z+y)+\frac{3y^2}{8\sigma^2 T}&\geq \frac{z^2}{2\sigma^2T}-\frac{|z+y|^2}{4\sigma^2 T}-R\gamma|z+y|+\frac{3y^2}{8\sigma^2 T}\geq -C_R
\end{align*}}
where $C_R$ is a constant that only depends on $R$ (and eventually on $\gamma,\sigma, T$). Thus, 
 $e^{\frac{z^2}{2\sigma^2T}+\gamma {\tilde g}(z+y)-\frac{3y^2}{8\sigma^2 T}} $ and $((\frac{z}{\sigma^2T}+\gamma  g(z+y))^2+\frac{1}{\sigma^2T}+\gamma  g' (z+y))e^{\frac{z^2}{2\sigma^2T}+\gamma {\tilde g}(z+y)-\frac{3y^2}{8\sigma^2 T}} $ are uniformly bounded by a constant depending only on $R$. 
Thus, $\E\left[e^{\frac{z^2}{2\sigma^2T}+\gamma {\tilde g}(z+\sigma B_T)}\right]$ is finite, {bounded by below by a positive constant depending only on $R$}, and $G_g(z)$ is twice continuously differentiable in $z$. 
  Additionally, we have that 
 \begin{align*} 
G''_g(z)&=\E\left[\left((\frac{z}{\sigma^2T}+\gamma  g(z+\sigma B_T))^2+\frac{1}{\sigma^2T}+\gamma g'(z+\sigma B_T)\right)e^{\frac{z^2}{2\sigma^2T}+\gamma {\tilde g}(z+\sigma B_T)}\right]\\
&\geq \frac{1}{\sigma^2T} \E\left[e^{\frac{z^2}{2\sigma^2T}+\gamma  g(0)(z+\sigma B_T)}\right] \geq \tilde C_{ R} 
\end{align*} 
for a second constant $\tilde C_{R}$. Finally we obtain that for all $g\in \cX_{l}$ 
 $$z\mapsto G_g(z)$$ is strongly convex with convexity constant $\tilde C$ only depending on $g(0)$.

 Let $g^n\in \cX_{l}$ converging to $g^0\in \cX_{l}$ in the sense of \eqref{eq:condconv}. Thus, $\{g^n(0)\}$ is uniformly bounded by some $R>0$. The boundedness of $e^{\frac{z^2}{2\sigma^2T}+\gamma {\tilde g}^n(z+y)-\frac{3y^2}{8\sigma^2 T}}$ and the Lipschitz continuity of the exponential function on sets bounded from above implies that  for all $|z|\leq R$ we have that 
 $$|e^{\frac{z^2}{2\sigma^2T}+\gamma {\tilde g}^n(z+y)-\frac{3y^2}{8\sigma^2 T}}-e^{\frac{z^2}{2\sigma^2T}+\gamma {\tilde g}^0(z+y)-\frac{3y^2}{8\sigma^2 T}}|\leq C_R | {\tilde g}^n(z+y)-{\tilde g}^0(z+y)|$$
 for some constant depending only on $R$. The assumed convergence of $g^n$ to $g^0$ and the fact that ${\tilde g}^n(0)={\tilde g}^0(0)=0$ implies that  $| {\tilde g}^n(z+y)-{\tilde g}^0(z+y)|\to 0$ for all $|z|\leq R$ and $y\in \R$. Thus, by a dominated convergence theorem 
 \begin{align*}
\left| \E\left[e^{\frac{z^2}{2\sigma^2T}+\gamma {\tilde g}^n(z+\sigma B_T)}-e^{\frac{z^2}{2\sigma^2T}+\gamma {\tilde g}^0(z+\sigma B_T)}\right]\right|\to 0. 
 \end{align*} 
 Similarly, 
 \begin{align*}
 &\left|\left(\frac{z}{\sigma^2 T}+\gamma g^n(z+y)\right)e^{\frac{z^2}{2\sigma^2T}+\gamma {\tilde g}^n(z+y)-\frac{3y^2}{8\sigma^2 T}}-\left(\frac{z}{\sigma^2 T}+\gamma g^0(z+y)\right)e^{\frac{z^2}{2\sigma^2T}+\gamma {\tilde g}^0(z+y)-\frac{3y^2}{8\sigma^2 T}}\right|\\
  &\leq C_R\left(\left|\frac{z}{\sigma^2 T}+\gamma g^n(z+y)\right|| {\tilde g}^n(z+y)-{\tilde g}^0(z+y)|+\gamma\left| g^n(z+y)- g^0(z+y)\right|\right).
  \end{align*}
 Since $|g^n(0)| \leq R$, $|z|\leq R$ and $(g^n)'\in [0,l]$, 
 we can bound 
 $$\left|\frac{z}{\sigma^2 T}+\gamma g^n(z+y)\right|\leq C_R(1+e^{\frac{y^2}{16\sigma^2 T}})$$
 so that thanks to the term $e^{-\frac{y^2}{8\sigma^2 T}}$ in \eqref{eq:intG} by a dominated convergence argument, 
 we have 
 $$(G_{g^n}(z),G'_{g^n}(z))\to(G_{g^0}(z),G'_{g^0}(z))$$ 
 uniformly on bounded sets of $z$.
 
  Since $\{g^n(0):n\geq 0\}$ is bounded, the family of functions 
 $$\left\{z\mapsto  \E\left[e^{\frac{z^2}{2\sigma^2T}+\gamma {\tilde g}^n(z+\sigma B_T)}\right]: n\geq 0\right\}$$
 are uniformly strongly convex. We denote $C>0$ a uniform strong convexity constant.  

{\it Step 2: Properties of the minimizers}

By strong convexity we now have that $z_n:=\tilde \chi^{{\tilde g}^n}$ exists as the unique minimizer 
$\argmin_{z\in \R}G_{g^n}(z)$.
By trivial estimates we have that 
\begin{align}\label{eq:boundZ}
&\left|G'_{g^n}(z_0)\right|=\left|\E\left[\left(\frac{z_0}{\sigma^2T}+\gamma g^n(z_0+\sigma B_T)\right)e^{\frac{z_0^2}{2\sigma^2T}+\gamma {\tilde g}^n(z_0+\sigma B_T)}\right]\right|\geq C|z_0-z_n|
\end{align}
As $n\to \infty$
\begin{align*}
&E\left[\left(\frac{z_0}{\sigma^2T}+\gamma  g^n(z_0+\sigma B_T)\right)e^{\frac{z^2_0}{2\sigma^2T}+\gamma {\tilde g}^n(z_0+\sigma B_T)}\right]\\
&\quad \quad \to E\left[\left(\frac{z_0}{\sigma^2T}+\gamma  g^0(z_0+\sigma B_T)\right)e^{\frac{z^2_0}{2\sigma^2T}+\gamma {\tilde g}^0(z_0+\sigma B_T)}\right]\\
&\quad\quad=G'_{g^0}(z_0)=0
\end{align*}
which implies thanks to \eqref{eq:boundZ} that $\tilde \chi^{{\tilde g}^n}\to \tilde \chi^{{\tilde g}^0}$.  This convergence in return easily implies $\tilde \Gamma^{{\tilde g}^n}\to \tilde \Gamma^{{\tilde g}^0}$.

Note also that the equality \eqref{eq:boundZ} with $g^0=0$ (and therefore $z_0=0$) yields that for all $g\in \cX_l$, 
\begin{align}
\left|\E\left[\gamma  g(\sigma B_T)e^{\gamma {\tilde g}(\sigma B_T)}\right]\right|\geq C|\tilde \chi^{{\tilde g}}|.
\end{align}
Therefore, $\tilde \chi^{{\tilde g}}$ is bounded by a constant that only depend on $\gamma,\sigma,T$ and $g(0)$. 
Since $\tilde \chi^{\tilde g}=\argmin_{z\in \R}G_g(z)$ injecting the bound of $\tilde \chi$ to $\tilde \Gamma^{\tilde g}=\min_{z\in \R}\frac{\ln G_g(z)}{\gamma}-\frac{(\chi^{{\tilde g}}(0,0))^2}{2\sigma^2\gamma T}$ we obtain also that 
$\tilde \Gamma^{\tilde g}$ is bounded by a constant depending on $g(0)$ (and eventually on $T,\sigma, \gamma$).

{\it Step 3: Proving $\chi^{\tilde g}(0,0)=\tilde \chi^{\tilde g}$ and $\Gamma^{\tilde g}(0,0)=\tilde \Gamma^{\tilde g}$ for $g\in\cX$ smooth enough.}
Thanks to the PDE \eqref{eq:pdep} and the dynamics of $\xi^{0,{\tilde g}}$ in \eqref{eq:sde}, 
for all ${\tilde g}\in C_{l,\a}$, 
$P(t,\xi^0_t)$ is a martingale and therefore, 
\begin{align*}
P(0,0)&=\E[g(\xi_T^0)]=\int_\R g(y)f_{{\tilde g}}(y)dy\\
&=\E\left[g(\chi^{{\tilde g}}(0,0)+\sigma B_T)e^{\gamma {\tilde g}(\chi^{{\tilde g}}(0,0)+\sigma B_T)-\gamma\Gamma^{\tilde g}(0,0)}\right].
\end{align*}
Given the expression of density $f_{\tilde g}$, \eqref{eq:fphi}, we also have
$$e^{\gamma \Gamma^{\tilde g}(0,0)}=\E\left[e^{\gamma {\tilde g}(\chi^{{\tilde g}}(0,0)+\sigma B_T)}\right].$$
Therefore, \eqref{eq:defzg} shows that $\chi^{{\tilde g}}(0,0)$ solves the equation 
$$\frac{\chi^{{\tilde g}}(0,0)}{\sigma^2 T}=\frac{\E\left[-\gamma g(\chi^{{\tilde g}}(0,0)+\sigma B_T)e^{\gamma {\tilde g}(\chi^{{\tilde g}}(0,0)+\sigma B_T)}\right]}{\E\left[e^{\gamma {\tilde g}(\chi^{{\tilde g}}(0,0)+\sigma B_T)}\right]}.$$
This is equivalent to 
$$G'_{{\tilde g}}(\chi^{{\tilde g}}(0,0))=0.$$
Since $G_{\tilde g}$ is strongly convex and admits $\tilde \chi^{\tilde g}$ as the unique minimizer, we have $\chi^{{\tilde g}}(0,0)=\tilde \chi^{\tilde g}$. {This equality and the identity 
$$e^{\gamma \Gamma^{\tilde g}(0,0)}=\E\left[e^{\gamma {\tilde g}(\chi^{{\tilde g}}(0,0)+\sigma B_T)}\right]=e^{-\frac{(\chi^{{\tilde g}}(0,0))^2}{2\sigma^2T}} G_g(\chi^{{\tilde g}}(0,0))=e^{\gamma \tilde \Gamma^{\tilde g}}.$$
directly implies $\Gamma^{{\tilde g}}(0,0)=\tilde \Gamma^{\tilde g}$ which conclude the proof of the Lemma. }
\end{proof}

\subsection{The Proof of Theorem \ref{thm:fixed}}\label{ss:prooffixed}

We first prove the theorem with Assumption \ref{assum:main} i). We use the Schauder fixed point theorem to prove this statement.  
Due to this assumption, $V(x)-\frac{\kappa x^2}{2}$ is convex. Fix $\gamma \in [0, \gamma_0)=[0,\frac{\sqrt{\kappa}}{4\sqrt{T\sigma^2}})$, $l=\frac{2}{\sqrt{\kappa\sigma^2 T}}$. We endow ${\cX_l}$ with the metric $|g|_{\cX_l}=|g(0)|+|g'|_{L^1(N(0,1))}$ where $|g'|_{L^1(N(0,1))}=\frac{1}{\sqrt{2\pi}} \int |g'(x)|e^{-\frac{x^2}{2}}dx$. Note that with this topology ${\cX_l}$ is a closed and convex subset of a Banach space.

For $\phi$ such that $\pa_\xi\phi\in {\cX_l}$,
we denote $F_\phi(x)=\int_{-\infty}^x f_\phi(y)dy$ (which only depends on $\pa_\xi\phi$ thanks to Lemma \ref{lem:rep}). 
We recall that the fixed point condition is
\begin{align}\label{eq:fixcond}
\textcolor{black}{\pa_{\xi}}\phi (\xi)=F_\nu^{-1}(F_\phi (\xi)).
\end{align}

{\it Step 1: Defining a continuous mapping.}
Recall that for $g\in \cX_l$, we denote $\tilde g$ its antiderivative that is $0$ at $0$. We define the following mapping 
\begin{align*}
M:{\cX_l}&\mapsto {\cX_l}\\
g&\mapsto M_g:=F_\nu^{-1}(F_{\tilde g}(\cdot))
\end{align*}
where
$$f_{\tilde g}(x)=\frac{1}{\sqrt{2\pi \sigma^2 T}}\exp\left(\gamma \tilde g(x)-\gamma \Gamma^{\tilde g}(0,0)-\frac{|\chi^{\tilde g}(0,0)-x|^2}{2\sigma^2T}\right)$$
and where by the Lemma \ref{lem:rep}
$$\chi^{\tilde g}(0,0)=\argmin_{z\in \R}\E\left[e^{\frac{z^2}{2\sigma^2T}+\gamma \tilde g(z+\sigma B_T)}\right],\,\Gamma^{ \tilde g}(0,0)=\min_{z\in \R}\frac{\ln G_g(z)}{\gamma}-\frac{(\chi^{{\tilde g}}(0,0))^2}{2\sigma^2\gamma T}.$$

Since the topology of $ {\cX_l}$ is stronger than the uniform convergence on compact sets, by the Lemma \ref{lem:rep} $\chi^{\tilde g}(0,0)$ and $\Gamma^{\tilde g}(0,0)$ are well defined and are continuous on $ {\cX_l}$. 

We now show that for all $g\in  {\cX_l}$, $M_g\in  {\cX_l}$. By definition, $M_g$ is the Brenier map pushing the measure with density $f_{\tilde g}$ onto $\nu$. 
Thanks to \eqref{eq:fphi} and the choice of $l$, the second derivative of $\ln f_{\tilde g}$ satisfies 
$$-\frac{1}{\sigma^2 T}\leq (\ln f_{\tilde g})''= \gamma g'-\frac{1}{\sigma^2 T}\leq -\frac{1}{2\sigma^2 T}.$$
Thus, we have $-(\ln f_{\tilde g})''\leq \frac{1}{\sigma^2 T}$. Additionally, by Assumption \ref{assum:main} $-(\ln p_\nu)''\geq\kappa$. Since, $M_g$ is the Brenier map pushing $f_{\tilde g}$ onto $p_\nu$,
thanks to the version of {Caffarelli's contraction theorem\footnote{Appendix \ref{s.optt} contains a summary of optimal transport based results needed in the paper.} in \cite[Corollary 6.1]{ko}}, we have
$$\kappa |M_g'|^2\leq \frac{1}{\sigma^2T}$$
and $M_g$ is $\frac{1}{\sqrt{\sigma^2 T \kappa}}$-Lipschitz continuous. By differentiating the definition of $M_g$ we have
$$M_g'(\xi)=\frac{f_{\tilde g}(\xi)}{p_\nu(M_g(\xi))}\geq 0$$ and thus $M_g\in \cX_l$.

We now show that $M$ is continuous. Let $g_n$ converging to $g_0$ in $\cX_l$. Due to the continuity of $\chi(0,0)$ and $\Gamma(0,0)$ and \eqref{eq:fphi2}, 
\begin{align*}
f_{\tilde g_n}(y)= \frac{1}{\sqrt{2\pi \sigma^2 T}}\exp\left(\gamma\int_0^y\int_0^rg'_n(s)dsdr-\gamma B_n-\gamma yA_n-\frac{y^2}{2\sigma^2T}\right)
\end{align*}
where
$A_n=-\frac{\chi^{\tilde g_n}(0,0)}{\gamma \sigma^2 T}-g_n(0)$ and $B_n=\Gamma^{\tilde g_n}(0,0)+\frac{\gamma\sigma^2 T}{2}(g_n(0)+A_n)^2$. 
Thus, 
$f_{\tilde g_n}\to f_{\tilde g_0}$ pointwise.
Additionally, the convergence of $g_n(0)$, the choice of $\gamma$ and $\cX_l$ implies that there exists a constant $C$ independent of $n$, so that 
\begin{align}\label{eq:boundlip}
\frac{1}{C\sqrt{4\pi \sigma^2 T}}\exp\left(-\frac{y^2}{4\sigma^2T}\right) \leq f_{\tilde g_n}(y)\leq   \frac{C}{\sqrt{4\pi \sigma^2 T}}\exp\left(-\frac{y^2}{4\sigma^2T}\right).
\end{align}

An application of the Dominated Convergence Theorem shows that $F_{\tilde g_n} \to F_{\tilde g}$ uniformly and by integrating \eqref{eq:boundlip} we obtain 
\begin{align}\label{eq:bbboundF}
\frac{F_{\sigma \sqrt{2T}}(y)}{C}\leq F_{\tilde g_n}(y)\leq  1- \frac{F_{\sigma \sqrt{2T}}(-y)}{C}.
\end{align} 
where $F_{\sigma \sqrt{2T}}$ is the cumulative distribution function of the normal distribution with variance $2\sigma^2 T$. 
Thus, we finally obtain that 
$M_{g_n}(y):=F_\nu^{-1}(F_{\tilde g_n}(y))$
converges pointwise to $M_{g_0}$. 

 Taking the derivative of its definition 
$$M'_{g_n}(y)=\frac{f_{\tilde g_n}(y)}{p_\nu(M_{g_n}(y))}$$
Thus, $M'_{g_n}$ converges pointwise to $M'_{g_0}$. Additionally, $ M'_{g_n}$ is bounded by $\frac{1}{\sqrt{\sigma^2 T \kappa}}$. Thus, by a simple Dominated convergence, 
$|M_{g_n}-M_g|_{\cX_l}$ goes to $0$ which is the convergence we wanted.

{\it Step 2: Image of bounded sets is relatively compact.}

We now show that image of bounded sets of $\cX_l$ are relatively compact sets. Let $g_n\in \cX_l$ be a bounded sequence.
Thus, similarly as above, $f_{\tilde g_n}$ is uniformly bounded and therefore $F_{\tilde g_n}$ is uniformly Lipschitz continuous and satisfy \eqref{eq:bbboundF} for a constant only depending on the bound of $\{|g_n|_{\cX_l}:n\}$.
Note that \eqref{eq:bbboundF} implies that 
\begin{align*}
F_\nu^{-1}\left(\frac{F_{\sigma \sqrt{2T}}(y)}{C}\right)\leq M_{\tilde g_n}(y)\leq  F_\nu^{-1}\left(1- \frac{F_{\sigma \sqrt{2T}}(-y)}{C}\right).
\end{align*} 
 Thus, given the assumption on $p_\nu$ and the boundedness of $f'_{\tilde g_n}$ on compact sets of $y$ uniformly in $n$, we have that $M_{g_n}$ are equicontinuous on compact sets of $y$. The family is also bounded. By Arzela Ascoli Theorem, there exists a subsequence of $g_n$ which we still denote $g_n$ so that the convergence
$$M'_{g_n}(\xi)\to h(\xi)\in [0,\frac{2}{\sqrt{\kappa \sigma^2 T}}]$$
is locally uniformly for some continuous function $h$. 
For $r\in \R$ to be determined, define 
$g_{h,r}(\xi)=r+\int_0^\xi h(x)dx\in \cX_l$. 
Note that thanks to Proposition \ref{prop:f}, 
$f_{\tilde g_{h,r}}$ (and $F_{\tilde g_{h,r}}$) in fact does not depend on $r$. 
Therefore, we can define 
$r_0:=M_{g_{h,r}}(0)$ and $ g^*(\xi):=g_{h,r_0}(\xi)$. 

Using one more time the fact that $F_{\tilde g}$ only depends on $g'$, and the fact that uniform convergence on compact sets is stronger than convergence in $L^1(N(0,1))$, we have that
$$M_{g_n}(0)\to r_0=g^*(0).$$
Thus, $M_{g_n}\to g^*$ in $\cX_l$
which is the compactness needed. 

Applying the Schauder fixed point theorem, we have the existence of $g$ satisfying 
\begin{align}\label{eq:fp1}
g(\xi)=F_\nu^{-1}(F_{\tilde g}(\xi)).
\end{align}

In order to finish the proof of the Lemma we need to improve the regularity of $g$ and show that $\tilde g\in C_{l,\a}$. 
We already have that $g$ is $\frac{1}{\sqrt{\sigma^2 T \kappa}}$-Lipschitz continuous. The equality \eqref{eq:fp1} yields that $g$ is continuously differentiable and its derivative satisfies
\begin{align}
g' (\xi)=\frac{f_{\tilde g}(\xi)}{p_\nu(g(\xi))}=\frac{1}{\sqrt{2\pi \sigma^2T}}e^{\gamma\tilde g(\xi)-\Gamma^{\tilde g}-\frac{|\chi^{\tilde g}-\xi|^2}{2\sigma^2 T}+V(g(\xi))}
\end{align}
and is bounded by $\frac{1}{\sqrt{\sigma^2 T \kappa}}$. Taking three more derivatives we have that $\tilde g\in C_{l,\a}$. 

We now prove the theorem with the Assumption \ref{assum:main} (ii). Under this assumption we cannot rely on \cite[Corollary 6.1]{ko}. 
However, due to the boundedness of the support of $\nu$, 
the mapping $M_g$ is bounded. Thus, one can easily show that
$$M'_{g}(\xi)=\frac{f_{\tilde g}(\xi)}{p_\nu(M_{g}(\xi))}$$
is uniformly bounded by constant that only depends on the support of $\nu$ and $\inf\{p_\nu(x):p_\nu(x)>0\}$. 

Thus, choosing $l$ large enough to ensure $M_g\in \cX_l$ and $\gamma>0$ small enough to ensure $\gamma l\in (0,\frac{1}{2\sigma^2 T})$, we can use the same fixed point argument as above.
 \begin{proof}[Proof of Lemma \ref{lem:measurability}]
  By the Lemma \ref{lem:quasi}, $P$ is Lipschitz continuous and $1-\gamma\sigma^2T l \leq \textcolor{black}{\pa_{\xi}}\chi\leq 1$. Thus, $\chi(t,\cdot)$ is invertible and its inverse is uniformly Lipschitz continuous in $\xi$. $P$ is also uniformly Lipschitz continuous in $\xi$.
  Therefore, by classical arguments, for any realization of the path of $(t,y)\in \Lambda$, one can establish the existence and uniqueness of a fixed point for the mapping
  $$(u(t))_{\{t\in [0,T]\}}\mapsto \left(\chi(0,0)+y(t)+\int_0^t {\gamma\sigma^2}P (r,u(r))dr\right)_{\{t\in [0,T]\}}$$
 which is the process $\xi$. 
 
 Additionally, by a direct application of the Gronwall's lemma, for two  $(t,y^1),(t,y^2)\in \Lambda$, the fixed points $\xi^1$ and $\xi^2$ satisfy
 $$|\xi^1(t)-\xi^2(t)|\leq C|y^1(t)-y^2(t)|+C\int_0^t|y^1(s)-y^2(s)|ds$$
 for some constant $C$. {Therefore $(t,y)\in \Lambda\mapsto\xi_t(y)$ is a continuous functional. It is also clear that $Y(t)$ is a continuous functional of $\{\xi(s)\}_{s\leq t}$.} Thus, both semimartingales $Y$ and $\xi(t,Y_\cdot)$ generate the same filtration. 
 
 Note that \eqref{eq:dynxi} yields that
$$\xi(t,y_\cdot)=\chi^{-1}\left(t,\chi(0,0)+y(t)+\int_0^t {\gamma\sigma^2}P (r,\xi(r,y_\cdot))dr\right).$$
Thus, using the regularity of $\chi$ and its inverse one can easily show that $\xi\in C^{1,2}(\Lambda)$ and a simple differentiation of $\xi(t,y_\cdot)$ in $y(t)$ and the fact that $\pa_y(\int_0^t {\gamma\sigma^2}P (r,\xi(r,y_\cdot))dr)=0$ yields
\begin{align}\label{eq:pathder}
\textcolor{black}{\pa_{y}}\xi(t,y_\cdot)=\frac{1}{\textcolor{black}{\pa_{\xi}}\chi(t,\xi(t,y_\cdot))}.
\end{align}
\end{proof}


\begin{appendix}
\numberwithin{equation}{subsection}
\section{Optimal transport}\label{s.optt}
We recall some well-known results from optimal transport theory. 
Let $\mu$ and $\nu$ be two absolutely continuous distributions on $\R$ with finite second moments. We say that a measurable map $f:\R\mapsto \R$, pushes $\mu$ forward to $\nu$ if for all Borel measurable subset $B$ of $\R$, we have 
$\nu(B)=\mu(f^{-1}(B))$. 

Thanks to the Brenier's theorem in \cite{b,m}, up to an additive constant, there exists a unique convex function $\phi$ so that $\pa_\xi \phi$ pushes $\mu$ forward to $\nu$. In this special case of one dimension $\pa_\xi \phi(x)=F_\nu^{-1}(F_\mu(x))$ for all $x\in \R$. We call the function $\phi$ the Brenier potential and $\pa_\xi \phi$ the Brenier map. 

Additionally, defining the convex conjugate $\phi^c(v)=\sup_\xi v\xi-\phi(\xi)$ of $\phi$, $\phi^c$ is also the unique convex function (up to an additive constant) whose derivative pushes $\nu$ forward to $\mu$. We also have $\pa_v(\phi^c)(v)=(\pa_\xi \phi)^{-1}(v)=F_\mu^{-1}(F_\nu(v))$ where the first equality holds in any dimension and the last equality is particular to the one dimensional transport problems. 

Assume that $\mu(dx)=e^{-W(x)}dx$ and $\nu=e^{-V(x)}dx$ where $V$ is $\kappa$ strongly convex in the sense that $V(x)-\frac{\kappa x^2}{2}$ is a convex function and $W$ twice differentiable. 
Then, thanks to \cite[Corollary 6.1]{ko}, the following second derivative estimate holds $\mu$-a.s
\begin{align}\label{ko:caf}
0\leq (\pa_{\xi\xi} \phi)^2\leq\frac{1}{\kappa} \inf\{y\in \R: y\geq (W'')_+,\, \mu-\mbox{ a.s.} \}.
\end{align}

\section{Auxiliary Results}\label{appn} 
{
\begin{prop}
{Assume Assumption \ref{assum:main}, and fix $\gamma\in (0,\gamma_1) $ as in Theorem \ref{thm:eq}. Then, any continuous semimartingale $X$, with $X_0=0$ satisfying the boundedness from below assumption of the realized gains \eqref{eq:bb} is admissible. }
\end{prop}

\begin{proof}
We now pick $X$ as above and assume that \eqref{eq:bb} holds. Proceeding as in the Proof of Theorem \ref{thm:eq} and instead of integrating the dynamics of $-\gamma d(\tilde v \chi^*(t,\xi_t)-\Gamma^*(t,\xi_t))$ and $
\gamma d\langle P,X\rangle_t$ on $[0,T]$ to obtain \eqref{eq:dcmp1}, we integrate them on $[0,t]$ to obtain that 
\begin{align}\notag
&e^{\int_0^t\gamma(\tilde v-P_s)\sigma dB_s-\int_0^t\frac{\gamma^2\sigma^2}{2}(\tilde v-P_s)^2 ds}\\
&=e^{-\gamma W_t-\gamma(\tilde v \chi^*(0,0)-\Gamma^*(0,0))+\gamma (\tilde v \chi^*(t,\xi_t)-\Gamma^*(t,\xi_t))-\frac{\tilde v^2\gamma^2\sigma^2t}{2}-\int_0^t\frac{\gamma\a_s^2{\pa_{\xi}}P^*(t,\xi_s)}{2{\pa_{\xi}}\chi^*(s,\xi_s)}ds}\notag\\
&\leq e^{-\gamma C(\tilde v)-\gamma(\tilde v \chi^*(0,0)-\Gamma^*(0,0)) +\gamma \sup_{(t,\xi)\in[0,T]\times \R}(\tilde v \chi^*(t,\xi)-\Gamma^*(t,\xi))}\label{eq:boundmg}
\end{align}
where as in \eqref{eq:bb}, $W_t=\int_0^t (\tilde v-P_s)dX_s-\langle X,P\rangle_t=(\tilde v-P_t)X_t+\int_0^t X_s dP_s$ is the realized gains up to time $t$ and $C(\tilde v)$ is the lower bound for this gain. 

We now show the weak maximum principle  
$$ \sup_{(t,\xi)\in[0,T]\times \R}(\tilde v \chi^*(t,\xi)-\Gamma^*(t,\xi))= \sup_{\xi\in \R}(\tilde v \chi^*(T,\xi)-\Gamma^*(T,\xi))= \sup_{\xi \in \R}(\tilde v \xi-\phi^*(\xi))=\phi^c(\tilde v)$$
for the function $(t,\xi)\mapsto \tilde v \chi^*(t,\xi)-\Gamma^*(t,\xi)$.
Define the function $\tilde \Gamma (t,\chi)=\Gamma^*(t,\chi^{-1}(t,\chi))$ where $\chi^{-1}(t,\cdot)$ is the inverse mapping of $\chi^*(t,\cdot)$. By differentiation we have that $$\pa_\chi \tilde \Gamma(t,\chi)=\frac{\pa_\xi \Gamma^*(t,\chi^{-1}(t,\chi))}{\pa_\xi\chi^*(t,\chi^{-1}(t,\chi))}=P(t,\chi^{-1}(t,\chi))$$
which is strictly increasing as a consequence of \eqref{eq:lip}.
Thus, $\tilde \Gamma$ is strictly convex in $\chi$. 
and for a given $t\in [0,T]$, the first order optimality condition and this convexity shows that  
$\sup_\xi\tilde v \chi^*(t,\xi)-\Gamma^*(t,\xi)=\sup_\chi(\tilde v\chi-\tilde \Gamma(t,\chi))$
is uniquely achieved at strict maximum $\xi$ satisfying $P^*(t,\xi)=\tilde v$. 
Assume now that $$\sup_{(t,\xi)\in[0,T]\times \R}\left(\tilde v \chi^*(t,\xi)-\Gamma^*(t,\xi)\right)=\sup_{t}\left(\tilde v \chi^*(t,(P^*)^{-1}(t,\tilde v))-\Gamma^*(t,(P^*)^{-1}(t,\tilde v))\right)$$
is achieved at $t<T$.  
Denoting $\xi^*=(P^*)^{-1}(t,\tilde v)$, the optimality at $(t,\xi^*)$ and strong convexity leads to $\pa_t (\tilde v \chi^*-\Gamma^*)(t,\xi^*)\leq 0$ and  $\pa^2_{\xi\xi} (\tilde v \chi^*-\Gamma^*)(t,\xi^*)< 0$. 
Injecting these inequalities and the equality $P^*(t,\xi^*)=\tilde v$ to \eqref{eq:system} we obtain 
\begin{align*}
 0&=\pa_t(\tilde v \chi^*-\Gamma^*)(t,\xi^*)+\frac{\sigma^2}{2(\pa_\xi\chi^*(t,\xi^*))^2}\pa_{\xi\xi}(\tilde v \chi^*-\Gamma^*)(t,\xi^*)-\frac{\gamma\sigma^2P^*(t,\xi^*)}{2}(2\tilde v-P^*(t,\xi^*))\\
 &<0-\frac{\gamma\sigma^2P^*(t,\xi^*)}{2}(2\tilde v-P^*(t,\xi^*))=-\frac{\gamma\sigma^2\tilde v^2}{2}.
 \end{align*}
We obtain a contradiction which implies the weak maximum principle 
$$\sup_{(t,\xi)\in[0,T]\times \R}\left(\tilde v \chi^*(t,\xi)-\Gamma^*(t,\xi)\right)=\phi^c(\tilde v)$$
as claimed. 
Thanks to \eqref{eq:boundmg}, we have the upper bound 
\begin{align*}
\sup_{t\in[0,T]}e^{\int_0^t\gamma(\tilde v-P_s)\sigma dB_s-\int_0^t\frac{\gamma^2\sigma^2}{2}(\tilde v-P_s)^2 ds}\leq \tilde C(\tilde v)
\end{align*}
for a finite deterministic function $\tilde C$. Since $\tilde v\in \cF_0$, 
the $\cF$-local martingale 
$$t\mapsto \frac{e^{\int_0^t\gamma(\tilde v-P_s)\sigma dB_s-\int_0^t\frac{\gamma^2\sigma^2}{2}(\tilde v-P_s)^2 ds}}{1+|\tilde C(\tilde v)|}$$
is uniformly bounded by $1$ and therefore is a $\cF$ martingale. 
This, easily implies that
$$t\mapsto {e^{\int_0^t\gamma(\tilde v-P_s)\sigma dB_s-\int_0^t\frac{\gamma^2\sigma^2}{2}(\tilde v-P_s)^2 ds}}$$
is a $\cF$-martingale
and
 for all $t\in[0,T]$, 
\begin{align*}
&\E\left[e^{\gamma\sigma\int_0^t(\tilde v- H(s,X_\cdot+Z_\cdot))dB_s-\frac{\gamma^2\sigma^2}{2}\int_0^t(\tilde v- H(s,X_\cdot+Z_\cdot))^2ds}|\cF_0\right]=1.
\end{align*}
This concludes the proof of the admissibility of $X$.

\end{proof}
}
\subsection{The PDE for $\Gamma$}\label{appA1}
 Since the computation is straightforward for $\chi$, we now show that the function $\Gamma$ defined by \eqref{eq:constintg} satisfies
\begin{align*}
 \textcolor{black}{\pa_{t}}\Gamma(t,\xi)+\frac{\sigma^2}{2(\textcolor{black}{\pa_{\xi}}\chi(t,\xi))^2}\textcolor{black}{\pa_{\xi\xi}}\Gamma(t,\xi)-\frac{\gamma\sigma^2P^2(t,\xi)}{2}=0.
\end{align*}
From Lemma \ref{lem:pdep}, we have,
\begin{align}
\textcolor{black}{\pa_{t}} P(t,\xi)=&\frac{-\sigma^2\textcolor{black}{\pa_{\xi\xi}}P(t,\xi)}{2(1-\gam\sigma^2\textcolor{black}{\pa_{\xi}}P(t,\xi)(T-t))^2}\\=&\frac{-\sigma^2}{2\gam\sigma^2(T-t)}\pa_\xi\Big(\frac{1}{1-\gamma\sigma^2 (T-t)\textcolor{black}{\pa_{\xi}}P (t,\xi)}\Big)
 \end{align}
 From Equation \eqref{eq:defzg}, we have,
 \begin{align}
\textcolor{black}{\pa_{\xi}}\chi(t,\xi)= 1-\gam\sigma^2(T-t) \textcolor{black}{\pa_{\xi}}P(t,\xi)
\end{align}
 Similarly from Equation \eqref{eq:constintg}, we have,
\begin{align*}
\textcolor{black}{\pa_{\xi}}\Gam(t,\xi)&=P(t,\xi)-\gam\sigma^2(T-t)P(t,\xi)\textcolor{black}{\pa_{\xi}}P(t,\xi)\\
\textcolor{black}{\pa_{\xi\xi}}\Gam(t,\xi)&=\textcolor{black}{\pa_{\xi}}P(t,\xi)-\gam\sigma^2(T-t)\big[P(t,\xi)\textcolor{black}{\pa_{\xi\xi}}P(t,\xi)+\textcolor{black}{(\pa_{\xi}P)^2}(t,\xi)\big]
\end{align*}
and therefore,
\begin{align*}
\frac{\sigma^2\textcolor{black}{\pa_{\xi\xi}}\Gam(t,\xi)}{2\textcolor{black}{(\pa_{\xi}\chi)}^2(t,\xi)}=\frac{\sigma^2\textcolor{black}{\pa_{\xi}}P(t,\xi)}{2\textcolor{black}{(\pa_{\xi}\chi)}^2(t,\xi)}-\frac{\gam\sigma^4(T-t)}{2\textcolor{black}{(\pa_{\xi}\chi)}^2(t,\xi)}\big[\textcolor{black}{(\pa_{\xi}}P)^2(t,\xi)+P(t,\xi)\textcolor{black}{\pa_{\xi\xi}}P(t,\xi)\big].
\end{align*}

By differentiating \eqref{eq:constintg}, we also have
\begin{align*}
\textcolor{black}{\pa_{t}}\Gam(t,\xi)=& \frac{-\sigma^2\textcolor{black}{\pa_{\xi}}P(t,0)}{2(1-\gam\sigma^2(T-t)\textcolor{black}{\pa_{\xi}}P(t,0))}-\frac{1}{2\gam(T-t)}\int_0^\xi\pa_\xi\Big(\frac{1}{1-\gam\sigma^2(T-t)\textcolor{black}{\pa_{\xi}}P(t,r)}\Big)dr\\&+\frac{\gam\sigma^2P^2(t,\xi)}{2}-\gam\sigma^2(T-t)P(t,\xi)\textcolor{black}{\pa_{t}}P(t,\xi)\\
=& \frac{-\sigma^2\textcolor{black}{\pa_{\xi}}P(t,0)}{2(1-\gam\sigma^2(T-t)\textcolor{black}{\pa_{\xi}}P(t,0))}\\
&-\frac{1}{2\gam(T-t)}\Big[\frac{1}{1-\gam\sigma^2(T-t)\textcolor{black}{\pa_{\xi}}P(t,\xi)}-\frac{1}{1-\gam\sigma^2(T-t)\textcolor{black}{\pa_{\xi}}P(t,0)}\Big]\\&+\frac{\gam\sigma^2P^2(t,\xi)}{2}-\gam\sigma^2(T-t)P(t,\xi)\textcolor{black}{\pa_{t}}P(t,\xi)\\
=& \frac{1}{2\gam(T-t)}\Big[1-\frac{1}{1-\gam\sigma^2(T-t)\textcolor{black}{\pa_{\xi}}P(t,\xi)}\Big]+\frac{\gam\sigma^2P^2(t,\xi)}{2}\\
&+\gam\sigma^2(T-t)P(t,\xi)\frac{\sigma^2}{2\textcolor{black}{(\pa_{\xi}\chi)}^2(t,\xi)}\textcolor{black}{\pa_{\xi\xi}}P\\
=& \frac{-\sigma^2\textcolor{black}{\pa_{\xi}}P(t,\xi)}{2\textcolor{black}{\pa_{\xi}}\chi(t,\xi)}+\frac{\gam\sigma^2P^2(t,\xi)}{2}+\frac{\gam\sigma^4(T-t)P(t,\xi)\textcolor{black}{\pa_{\xi\xi}}P}{2\textcolor{black}{(\pa_{\xi}\chi)}^2(t,\xi)}\\
\end{align*}
Plugging in the values in Equation \eqref{eq:system}, we have,
\begin{align*}
&\textcolor{black}{\pa_{t}}\Gam(t,\xi)+\frac{\sigma^2\textcolor{black}{\pa_{\xi\xi}}\Gam(t,\xi)}{2\textcolor{black}{(\pa_{\xi}\chi)}^2(t,\xi)}-\frac{\gam\sigma^2P^2(t,\xi)}{2}\\
=&\frac{-\sigma^2\textcolor{black}{\pa_{\xi}}P(t,\xi)}{2\textcolor{black}{\pa_{\xi}}\chi(t,\xi)}+\frac{\gam\sigma^2P^2(t,\xi)}{2}+\frac{\gam\sigma^4(T-t)P(t,\xi)\textcolor{black}{\pa_{\xi\xi}}P(t,\xi)}{2\textcolor{black}{(\pa_{\xi}\chi)}^2(t,\xi)}\\&+\frac{\sigma^2\textcolor{black}{\pa_{\xi}}P(t,\xi)}{2\textcolor{black}{(\pa_{\xi}\chi)}^2(t,\xi)}-\frac{\gam\sigma^4(T-t)}{2\textcolor{black}{(\pa_{\xi}\chi)}^2(t,\xi)}\big[P(t,\xi)\textcolor{black}{\pa_{\xi\xi}}P(t,\xi)+\textcolor{black}{(\pa_{\xi}P)}^2(t,\xi)\big] -\frac{\gam\sigma^2P^2(t,\xi)}{2}\\
=&\frac{\sigma^2\textcolor{black}{\pa_{\xi}}P(t,\xi)}{2\textcolor{black}{(\pa_{\xi}\chi)}^2(t,\xi)}\big[1-\textcolor{black}{\pa_{\xi}}\chi(t,\xi)\big]-\frac{\gam\sigma^4(T-t)}{2\textcolor{black}{(\pa_{\xi}\chi)}^2(t,\xi)}\textcolor{black}{(\pa_{\xi}P)}^2(t,\xi)\\
=&\frac{\sigma^2}{2\textcolor{black}{(\pa_{\xi}\chi)}^2(t,\xi)}\big[\textcolor{black}{\pa_{\xi}}P(t,\xi)-\textcolor{black}{\pa_{\xi}}P(t,\xi)\textcolor{black}{\pa_{\xi}}\chi(t,\xi)-\gam\sigma^2(T-t)\textcolor{black}{(\pa_{\xi}P)}^2(t,\xi)\big]\\
=&\frac{\sigma^2}{2\textcolor{black}{(\pa_{\xi}\chi)}^2(t,\xi)}\big[\textcolor{black}{\pa_{\xi}}P(t,\xi)-\textcolor{black}{\pa_{\xi}}P(t,\xi)\big(1-\gam\sigma^2(T-t) \textcolor{black}{\pa_{\xi}}P(t,\xi)\big)-\gam\sigma^2(T-t)\textcolor{black}{(\pa_{\xi}P)}^2(t,\xi)\big]= 0.
\end{align*}
\subsection{Price impact dynamics}\label{appA2}
Recall that
\begin{align}\label{eq:lambda}
\lambda(t,\xi_t^0)=\frac{\textcolor{black}{\pa_{\xi}}P(t,\xi_t^0)}{\textcolor{black}{\pa_{\xi}}\chi(t,\xi_t^0)}
\end{align}
We also have,
\begin{align}
\xi^0_t=\int_0^t\frac{\sigma}{(1-\gamma\sigma^2 (T-r)\textcolor{black}{\pa_{\xi}}P (r,\xi^0_r))} dB_r=\int_0^t\frac{\sigma}{\textcolor{black}{\pa_{\xi}}\chi (r,\xi^0_r)} dB_r
\end{align}
From \eqref{eq:sde}, we have:
\begin{align}
d\xi_t^0=\frac{\sigma}{\textcolor{black}{\pa_{\xi}}\chi(t,\xi^0_t)} dB_t
\end{align}
and
\begin{align}
d\langle\xi_t^0\rangle=\frac{\sigma^2}{\textcolor{black}{(\pa_{\xi}\chi)}^2 (t,\xi^0_t)} dt
\end{align}
Applying Ito's formula to $\lambda(t,\xi_t^0)$ in Equation\eqref{eq:lambda}, we get the following:
\begin{align}\label{eq:ito}
d\lambda(t,\xi_t^0)=&\textcolor{black}{\pa_{t}}\lambda(t,\xi_t^0)dt + \textcolor{black}{\pa_{\xi}}\lambda(t,\xi_t^0)d\xi_t^0 + \frac{1}{2}\textcolor{black}{\pa_{\xi\xi}}\lambda(t,\xi_t^0) d\langle\xi_t^0\rangle\notag\\
=&\Big[\textcolor{black}{\pa_{t}}\lambda(t,\xi_t^0)+\frac{\sigma^2}{2\textcolor{black}{(\pa_{\xi}\chi)}^2 (t,\xi^0_t)}\textcolor{black}{\pa_{\xi\xi}}\lambda(t,\xi_t^0)\Big]dt+\Big[\frac{\sigma}{\textcolor{black}{\pa_{\xi}}\chi(t,\xi^0_t)}\textcolor{black}{\pa_{\xi}}\lambda(t,\xi_t^0)\Big] dB_t
\end{align}
Here $\textcolor{black}{\pa_{\xi}}P=R$ satisfies \eqref{eq:pder}. By a direct computation, 
\begin{align}
\textcolor{black}{\pa_{t}}\lambda(t,\xi_t^0)=&\frac{\pa}{\pa t}\Big[\frac{\textcolor{black}{\pa_{\xi}}P(t,\xi_t^0)}{\textcolor{black}{\pa_{\xi}}\chi(t,\xi_t^0)}\Big]=\frac{ \textcolor{black}{\pa_{t}}R(t,\xi_t^0)-\gamma\sigma^2 R^2(t,\xi_t^0)}{\Big[1-\gamma\sigma^2(T-t)R(t,\xi_t^0)\Big]^2}\label{eq:l_t}
\\\notag\\
\textcolor{black}{\pa_{\xi}}\lambda(t,\xi_t^0)=&\frac{\pa}{\pa \xi}\Big[\frac{\textcolor{black}{\pa_{\xi}}P(t,\xi_t^0)}{\textcolor{black}{\pa_{\xi}}\chi(t,\xi_t^0)}\Big]=\frac{\textcolor{black}{\pa_{\xi}}R(t,\xi_t^0)}{\Big[1-\gamma\sigma^2(T-t)R(t,\xi_t^0)\Big]^2}\label{eq:l_xi}\\\notag\\
\textcolor{black}{\pa_{\xi\xi}}\lambda(t,\xi_t^0)=&\frac{\pa}{\pa \xi}\bigg[\frac{\textcolor{black}{\pa_{\xi}}R(t,\xi_t^0)}{\textcolor{black}{(\pa_{\xi}\chi)}^2(t,\xi_t^0)}\bigg]\notag\\
=&\frac{1}{\textcolor{black}{(\pa_{\xi}\chi)}^3(t,\xi_t^0)}\Big[\textcolor{black}{\pa_{\xi}}\chi(t,\xi_t^0)\textcolor{black}{\pa_{\xi\xi}}R(t,\xi_t^0)+2\gamma\sigma^2(T-t)\textcolor{black}{(\pa_{\xi}R)}^2(t,\xi_t^0)\Big]\label{eq:l_xixi}
\end{align}
Injecting Equations \eqref{eq:l_t}, \eqref{eq:l_xi} and  \eqref{eq:l_xixi} in Equation\eqref{eq:ito}, we get,
\begin{align}
d\lambda(t,\xi_t^0)=&\bigg[\frac{ \textcolor{black}{\pa_{t}}R(t,\xi_t^0)-\gamma\sigma^2 R^2(t,\xi_t^0)}{\textcolor{black}{(\pa_{\xi}\chi)}^2(t,\xi_t^0)}\notag\\
&+\frac{\sigma^2}{\textcolor{black}{(\pa_{\xi}\chi)}^2(t,\xi_t^0)}.\frac{1}{\textcolor{black}{(\pa_{\xi}\chi)}^3(t,\xi_t^0)}\Big(\textcolor{black}{\pa_{\xi}}\chi(t,\xi_t^0)\textcolor{black}{\pa_{\xi\xi}}R(t,\xi_t^0)+2\gamma\sigma^2(T-t)\textcolor{black}{(\pa_{\xi}R)}^2(t,\xi_t^0)\Big)\bigg]dt\notag\\
&+\frac{\sigma}{\textcolor{black}{\pa_{\xi}}\chi(t,\xi_t^0)}.\frac{\textcolor{black}{\pa_{\xi}}R(t,\xi_t^0)}{\textcolor{black}{(\pa_{\xi}\chi)}^2(t,\xi_t^0)} dB_t\notag\\
=&\Big[\textcolor{black}{\pa_{t}} R(t,\xi_t^0)-\gamma\sigma^2 R^2(t,\xi_t^0)+\frac{\sigma^2\textcolor{black}{\pa_{\xi\xi}}R(t,\xi_t^0)}{2\textcolor{black}{(\pa_{\xi}\chi)}^2(t,\xi_t^0)}+\frac{\gamma\sigma^4(T-t)\textcolor{black}{(\pa_{\xi}R)}^2(t,\xi_t^0)}{\textcolor{black}{(\pa_{\xi}\chi)}^3(t,\xi_t^0)}\Big]\frac{dt}{\textcolor{black}{(\pa_{\xi}\chi)}^2(t,\xi_t^0)}\notag\\&+\frac{\sigma \textcolor{black}{\pa_{\xi}}R(t,\xi_t^0)}{\textcolor{black}{(\pa_{\xi}\chi)}^3(t,\xi_t^0)} dB_t
\end{align}
Using Equation \eqref{eq:pder}, we finally get,
\begin{align*}
d\lambda(t,\xi_t^0)&=\frac{-\gamma\sigma^2 R^2(t,\xi_t^0)}{\textcolor{black}{(\pa_{\xi}\chi)}^2(t,\xi_t^0)} dt + \frac{\sigma \textcolor{black}{\pa_{\xi}}R(t,\xi_t^0)}{\textcolor{black}{(\pa_{\xi}\chi)}^3(t,\xi_t^0)} dB_t\\
&=-\gamma\sigma^2 \lambda^2(t,\xi_t^0) dt + \frac{\sigma \textcolor{black}{\pa_{\xi}}R(t,\xi_t^0)}{\textcolor{black}{(\pa_{\xi}\chi)}^3(t,\xi_t^0)} dB_t
\end{align*}

\subsection{Market depth dynamics}\label{appA3}
The price impact is defined as
\begin{align}
\lambda(t,\xi)=\frac{\textcolor{black}{\pa_{\xi}}P(t,\xi)}{\textcolor{black}{\pa_{\xi}}\chi(t,\xi)}.
\end{align}
Let, $\zeta$ be the market depth 
\begin{align}\label{eq:zeta}
\zeta(t,\xi_t)=\frac{1}{\lambda(t,\xi_t)}=\frac{\textcolor{black}{\pa_{\xi}}\chi(t,\xi_t)}{\textcolor{black}{\pa_{\xi}}P(t,\xi_t)}
\end{align}
From the problem formulation, we have,
\begin{align}
    dY_t=&dX_t+dZ_t\notag\\
    =&\theta^{\tilde v}(t,\xi_t)dt +\sigma dB_t\label{eq:dY}\\
    d\langle Y\rangle_t=&\sigma^2 dt\label{eq:dYdY}
\end{align}
Thus plugging Equation\eqref{eq:dY} and \eqref{eq:dYdY} into Equation\eqref{eq:dxi}, we get,
\begin{align}
    d\xi_t=\frac{dY_t}{\textcolor{black}{\pa_{\xi}}\chi(t,\xi_t)}
\end{align}and
\begin{align}
d\langle \xi\rangle_t=&\frac{1}{\textcolor{black}{(\pa_{\xi}\chi)}^2(t,\xi_t)}d\langle Y\rangle_t\notag\\
=&\frac{\sigma^2 dt}{\textcolor{black}{(\pa_{\xi}\chi)}^2(t,\xi_t)}
\end{align}
Applying Ito's formula to $\zeta(t,\xi_t)$ in Equation\eqref{eq:zeta}, we get the following:
\begin{align}\label{eq:itos}
d\zeta(t,\xi_t)=&\textcolor{black}{\pa_{t}}\zeta(t,\xi_t)dt + \textcolor{black}{\pa_{\xi}}\zeta(t,\xi_t)d\xi_t + \frac{1}{2}\textcolor{black}{\pa_{\xi\xi}}\zeta(t,\xi_t) d\langle\xi\rangle_t\notag\\
=&\Big[\textcolor{black}{\pa_{t}}\zeta(t,\xi_t)+\frac{\sigma^2}{2\textcolor{black}{(\pa_{\xi}\chi)}^2(t,\xi_t)}\textcolor{black}{\pa_{\xi\xi}}\zeta(t,\xi_t)\Big]dt+\Big[\frac{1}{\textcolor{black}{\pa_{\xi}}\chi (t,\xi_t)}\textcolor{black}{\pa_{\xi}}\zeta(t,\xi_t)\Big] dY_t
\end{align}
Here $\textcolor{black}{\pa_{\xi}}P=R$ which satisfies \eqref{eq:pder}. By a direct computation, 
\begin{align}
\textcolor{black}{\pa_{t}}\zeta(t,\xi_t)=&\frac{\pa}{\pa t}\Big[\frac{\textcolor{black}{\pa_{\xi}}\chi(t,\xi_t)}{R(t,\xi_t)}\Big]=\frac{\gamma\sigma^2R^2(t,\xi_t)- \textcolor{black}{\pa_{t}}R(t,\xi_t)}{R^2(t,\xi_t)}\label{eq:z_t}\\
\textcolor{black}{\pa_{\xi}}\zeta(t,\xi_t)=&\frac{\pa}{\pa \xi}\Big[\frac{\textcolor{black}{\pa_{\xi}}\chi(t,\xi_t)}{R(t,\xi_t)}\Big]=\frac{-\textcolor{black}{\pa_{\xi}}R(t,\xi_t)}{R^2(t,\xi_t)}\label{eq:z_xi}\\
\textcolor{black}{\pa_{\xi\xi}}\zeta(t,\xi_t)=&\frac{\pa}{\pa \xi}\Big[\frac{-\textcolor{black}{\pa_{\xi}}R(t,\xi_t)}{R^2(t,\xi_t)}\Big]=\frac{2\textcolor{black}{(\pa_{\xi}R)}^2(t,\xi_t)-R(t,\xi_t)\textcolor{black}{\pa_{\xi\xi}}R(t,\xi_t)}{R^3(t,\xi_t)}\label{eq:z_xixi}
\end{align}
Injecting Equations \eqref{eq:z_t}, \eqref{eq:z_xi} and  \eqref{eq:z_xixi} in Equation\eqref{eq:itos}, we get,
\begin{align}
d\zeta(t,\xi_t)=&\bigg[\frac{\gamma\sigma^2R^2(t,\xi_t)- \textcolor{black}{\pa_{t}}R(t,\xi_t)}{R^2(t,\xi_t)}+\frac{\sigma^2}{2\textcolor{black}{(\pa_{\xi}\chi)}^2(t,\xi_t)}.\frac{\big(2\textcolor{black}{(\pa_{\xi}R)}^2(t,\xi_t)-R(t,\xi_t)\textcolor{black}{\pa_{\xi\xi}}R(t,\xi_t)\big)}{R^3(t,\xi_t)}\bigg]dt\notag\\
&+\notag\Big[\frac{1}{\textcolor{black}{\pa_{\xi}}\chi (t,\xi_t)}.\frac{-\textcolor{black}{\pa_{\xi}}R(t,\xi_t)}{R^2(t,\xi_t)}\Big] dY_t
\\=&\bigg[\gamma\sigma^2-\frac{1}{R^2(t,\xi_t)}\Big( \textcolor{black}{\pa_{t}}R(t,\xi_t)+\frac{\sigma^2\textcolor{black}{\pa_{\xi\xi}}R(t,\xi_t)}{2\textcolor{black}{(\pa_{\xi}\chi)}^2(t,\xi_t)}\Big)+\frac{\sigma^2\textcolor{black}{(\pa_{\xi}R)}^2(t,\xi_t)}{\textcolor{black}{(\pa_{\xi}\chi)}^2(t,\xi_t)R^3(t,\xi_t)}\bigg]dt\notag\\-&\Big[\frac{\textcolor{black}{\pa_{\xi}}R(t,\xi_t)}{R^2(t,\xi_t)\textcolor{black}{\pa_{\xi}}\chi (t,\xi_t)}\Big] dY_t
\end{align}
Using Equation \eqref{eq:pder}, we finally get,
\begin{align}
d\zeta(t,\xi_t)=&\bigg[\gamma\sigma^2+\frac{1}{R^2(t,\xi_t)}.\frac{\gamma\sigma^4(T-t)\textcolor{black}{(\pa_{\xi}R)}^2(t,\xi_t)}{\textcolor{black}{(\pa_{\xi}\chi)}^3(t,\xi_t)}+\frac{\sigma^2\textcolor{black}{(\pa_{\xi}R)}^2(t,\xi_t)}{\textcolor{black}{(\pa_{\xi}\chi)}^2(t,\xi_t)R^3(t,\xi_t)}\bigg]dt\notag\\ -&\Big[\frac{\textcolor{black}{\pa_{\xi}}R(t,\xi_t)}{R^2(t,\xi_t)\textcolor{black}{\pa_{\xi}}\chi(t,\xi_t)}\Big] dY_t
\end{align}
\end{appendix}



\section*{Acknowledgements}

The authors would like to thank Gregoire Loeper for fruitful discussions. They also thank
two anonymous referees for their valuable remarks.

\bibliographystyle{imsart-number} 
\bibliography{ref}       


\end{document}